\documentclass[11pt]{amsart}

\usepackage{enumerate}
\usepackage{physics,esdiff}

\usepackage[bookmarks=true,
bookmarksnumbered=true, breaklinks=true,
pdfstartview=FitH, hyperfigures=false,
plainpages=false, naturalnames=true,
colorlinks=true,
pdfpagelabels]{hyperref}
\hypersetup{
  pdftitle={Minkowski Existence},
  pdfauthor={Dylan, Liudmyla}
}

\newtheorem{theorem}{Theorem}[section]
\newtheorem{lemma}[theorem]{Lemma}
\newtheorem{corollary}[theorem]{Corollary}
\newtheorem{proposition}[theorem]{Proposition}

\theoremstyle{definition}
\newtheorem{definition}[theorem]{Definition}
\newtheorem{example}[theorem]{Example}

\theoremstyle{remark}
\newtheorem{remark}[theorem]{Remark}
\newtheorem{question}[theorem]{Question}

\numberwithin{equation}{section}

\newcommand{\R}{\mathbb{R}} 
\newcommand{\capa}[2]{\operatorname{Cap}_{#1}(#2)}
\newcommand{\s}{\mathbb{S}}
\newcommand{\conbod}{\mathcal{K}^n}

\newcommand{\vol}{\text{\rm Vol}}
\newcommand{\ct}[2]{\mathcal{C}({#1})({#2})}

\title[Weighted Minkowski's Existence Theorem]{Weighted Minkowski's Existence Theorem \\ and Projection Bodies}

\author{Liudmyla Kryvonos}
\address{Department of Mathematics\\
Vanderbilt University\\
Nashville, TN 37240\\}
\email{liudmyla.kryvonos@Vanderbilt.Edu}

\author{Dylan Langharst}
\address{Department of Mathematical Sciences\\
Kent State University\\
Kent, OH 44242\\}
\email{dlanghar@kent.edu}
\thanks{The second named author was supported in part by the U.S. National Science Foundation Grant DMS-2000304 and the United States - Israel Binational Science Foundation (BSF)}



\begin{document}
\begin{abstract}
The Brunn-Minkowski Theory has seen several generalizations over the past century. Many of the core ideas have been generalized to measures. With the goal of framing these generalizations as a weighted Brunn-Minkowski theory, we prove the Minkowski existence theorem for a large class of Borel measures with continuous density, denoted by $\Lambda^n$: for $\nu$ a finite, even Borel measure on the unit sphere and even $\mu\in\Lambda^n$, there exists a symmetric convex body $K$ such that
$$d\nu(u)=c_{\mu,K}dS^{\mu}_{K}(u),$$ 
where $c_{\mu,K}$ is a quantity that depends on $\mu$ and $K$ and $dS^{\mu}_{K}(u)$ is the surface area-measure of $K$ with respect to $\mu$. Examples of measures in $\Lambda^n$ are homogeneous measures (with $c_{\mu,K}=1$) and probability measures with radially decreasing densities (e.g. the Gaussian measure). We will also consider weighted projection bodies $\Pi_\mu K$ by classifying them and studying the isomorphic Shephard problem: if $\mu$ and $\nu$ are even, homogeneous measures with density and $K$ and $L$ are symmetric convex bodies such that
$\Pi_{\mu} K \subset \Pi_{\nu} L$, then can one find an optimal quantity $\mathcal{A}>0$ such that $\mu(K)\leq \mathcal{A}\nu(L)$? Among other things, we show that, in the case where $\mu=\nu$ and $L$ is a projection body, $\mathcal{A}=1$.
\end{abstract}

\maketitle

\section{Introduction}
A convex body is a convex, compact set with nonempty interior. The set of convex bodies in $\R^n$ (the $n$-dimensional Euclidean space equipped with its usual structure) will be denoted as $\conbod$.  An important subset of $\conbod$ is the set of convex bodies containing the origin in their interior, denoted $\conbod_0.$ For $K\in\conbod_0$, we say $K$ is centrally symmetric, or just symmetric, if $K=-K.$ Convex bodies are often studied through their support function, which is given by, for $K\in\conbod,$ $$h_K(x)=\sup_{y\in K}\langle x, y\rangle.$$ 
We say a set $L$ with $0 \in \text{int}(L)$ is star-shaped if every line passing through the origin crosses the boundary of $L$ exactly twice. We say $L$ is a star body 
if it is a compact, star-shaped set whose radial function $\rho_L:\R^n\setminus \{0\} \to \R,$ given by $\rho_L(y)=\sup\{\lambda:\lambda y\in L\},$ is continuous. Furthermore, for $K\in\conbod_0,$ the \textit{Minkowski functional} of $K$ is defined to be $\|y\|_K=\rho^{-1}_K(y)=\inf\{r>0:y\in rK\}.$ The Minkowski functional $\|\cdot\|_K$ of $K\in\conbod_0$ is a norm on $\R^n$ if $K$ is symmetric.

For $\Omega\subset\R^n$, the volume of $\Omega$ is its $d$-dimensional Lebesgue measure, denoted $\vol_d(\Omega)$, where $d$ is the dimension of the least affine subspace of $\R^n$ containing $\Omega$. The closed unit ball in $\R^n$ will be denoted by $B_2^n$; its volume will be written  as $\kappa_n$. If we denote the boundary of $B_2^n$ by $\s^{n-1}$, then a convex body $K\in\conbod$ can also be studied through its surface area measure: for every Borel $A \subset \s^{n-1},$ $$S_K(A)=\mathcal{H}^{n-1}(n^{-1}_K(A)),$$ where $\mathcal{H}^{n-1}$ is the $(n-1)$-dimensional Hausdorff measure and $n_K:\partial K \rightarrow \s^{n-1}$ is the Gauss map, which associates an element $y$ of $\partial K$ with its outer unit normal. For almost all $x\in\partial K$, $n_K(x)$ is well-defined (i.e. $x$ has a single outer unit normal); we shall denote $$\partial^\prime K = \partial K\setminus \{x: n_K(x) \text{ is not well-defined}\}.$$ A convex body is said to be \textit{strictly convex} if $\partial K$ does not contain a line segment, in which case the Gauss map is an injection between $\partial^\prime K$ and $\s^{n-1}$. The Gauss map is related to the support function as follows: for a fixed $u\in\s^{n-1}$, $\nabla h_K(u)$ exists if, and only if, $n_K^{-1}(u)$ is a single point $x\in\partial K,$ and, furthermore, $\nabla h_K(u)=x=n^{-1}_K(u)$ \cite[Corollary 1.7.3]{Sh1}. Hence, $K$ is strictly convex if, and only if, $h_K\in C^1$ \cite[Page 115]{Sh1}.

If $h_K$ is of class $C^2$, then one has $dS_K(\theta)$ is absolutely continuous with respect to the spherical Lebesgue measure \cite[Corollary 2.5.3]{Sh1}: $$dS_K(\theta) = f_K(\theta) d\theta,$$
	where $f_K \colon \s^{n-1} \to \R$ is the \textit{curvature function} of $K$, and is merely the reciprocal of the Gauss curvature as a function of the outer unit normal. The class $C^2_+$ is used to denote those convex bodies whose support functions are $C^2$ and have positive radii of curvature (this is equivalent to definition from differential geometry, see \cite[Page 120]{Sh1}). From \cite[Theorem 2.7.1]{Sh1} every convex body can be approximated by convex bodies that are $C^2_+.$

Given a finite Borel measure $\mu$ on $\s^{n-1}$, one may ask: does there exist a unique (up to translations) convex body $K$ such that $dS_K=d\mu$? Minkowski's existence theorem \cite[p. 455]{Sh1} shows that if $\mu$ satisfies the following two conditions, then the answer is yes:
\begin{enumerate}
    \item The measure $\mu$ is not concentrated on any great hemisphere, that is
    \label{con_1}
    $$\int_{\s^{n-1}}|\langle \theta, \xi \rangle|d\mu(\xi) > 0 \quad \text{for all } \theta\in\s^{n-1}.$$
    \item The measure is centered, that is
    $$\int_{\s^{n-1}}\xi d\mu(\xi)=0.$$
    \label{con_2}
\end{enumerate}
Minkowski's existence theorem was proven by Aleksandrov using the following schema \cite{AL}. For every positive $f\in C(\s^{n-1})$, the \textit{Wulff shape} of $f$ is the convex body given by
	\begin{equation}
	    [f]=\{x\in\R^n:\langle x,u\rangle\leq f(u) \; \forall u \; \in \s^{n-1}\}.
	\end{equation}
	One has that, for $K\in\conbod_0$, $[h_K]=K.$ Since $f$ is positive, $[f]\in\conbod_0$. Furthermore, if $f$ is even, then $[f]$ is symmetric. Next, for $f\in C(\s^{n-1})$, Aleksandrov defined a perturbation of $K$ to be the Wulff shape of the function
	$$h_t(u)=h_K(u)+tf(u),$$
	where $t\in (-\delta,\delta)$, $\delta$ small enough so that $h_t$ is positive for all $u$. From here, Aleksandrov showed his variational formula:
	$$\diff{\vol_n([h_{t}])}{t}\bigg|_{t=0}=\lim_{t\to 0}\frac{\vol_n([h_t])-\vol_n(K)}{t}=\int_{\s^{n-1}}f(u)dS_K(u);$$
	setting $f=h_L$ for $L\in\conbod$ then yields that the \textit{mixed volume} of $(n-1)$ copies of $K$ and $L$ \cite{gardner_book} is 
	\begin{equation}
	    V(K,L)=\frac{1}{n}\lim_{\epsilon\to0}\frac{\vol_n(K+\epsilon L)-\vol_n(K)}{\epsilon}=\frac{1}{n}\int_{\s^{n-1}}h_L(u)dS_K(u),
	    \label{eq:mixed}
	\end{equation}
	and, furthermore, mixed volume satisfies \textit{Minkowski's inequality}:
	\begin{equation}
	V(K,L)^n\geq \vol_n(K)^{n-1}\vol_n(L).
	\label{eq:min_ineq}
	\end{equation}
Using these ingredients, Aleksandrov settled the question of Minkowski's existence theorem. By setting $L=K$ in Equation~\ref{eq:mixed}, one has that the volume of $K\in\conbod$ is given by $$\vol_n(K)=\frac{1}{n}\int_{\s^{n-1}}h_K(u)dS_K(u).$$

Next, consider a homogeneous function of order 1, $g:\R^n\to\R^+$. It can easily be shown that $g$ is a norm on $\R^n$ if, and only if, its unit ball $L=\{x:g(x)\leq 1\}$ is a symmetric convex body; given symmetric $K\in\conbod_0$, the Minkowski functional is precisely the associated norm. For $K\in\conbod_0,$ the dual body of $K$ is given by $$K^\circ=\left\{x\in \R^n: h_K(x) \leq 1\right\}.$$ We see that $K^\circ$ is the unit ball of $h_K(x)$, and so, by definition $h_K(x)=\|x\|_{K^\circ}$. For $K \in \conbod$, we denote the orthogonal projection of $K$ onto a linear subspace $H$ as $P_H K$. \textit{Minkowski's projection formula} then states \cite{Sh1}: for $\theta\in\mathbb{S}^{n-1}$, one has that 
\begin{equation}
   \vol_{n-1}\left(P_{\theta^{\perp}}K\right) =\frac{1}{2}\int_{\s^{n-1}}|\langle \theta,u \rangle| dS_K(u),
\label{eq:projection}
	\end{equation}
where  $\theta^{\perp}=\{x\in\R^n:\langle \theta,x\rangle =0\}$ is the subspace orthogonal to $\theta \in \s^{n-1}$. If one extends Equation~\ref{eq:projection} to all of $\R^n$ $1$-homogeneously (merely replace $\theta$ in the above with arbitrary $x$), then we see it is a norm on $\R^n$. Using the above arguments concerning unit balls of norms and the relation between the Minkowski functional of a convex body and the support function of its dual, one can then define the \textit{projection body of K}, denoted $\Pi K$, as the symmetric convex body whose support function is given by 
\begin{equation}
    h_{\Pi K}(\theta)=\vol_{n-1}(P_{\theta^{\perp}}K) =\frac{1}{2}\int_{\s^{n-1}}|\langle \theta,u \rangle| dS_K(u).
    \label{proj_support}
\end{equation}

The projection body operator $\Pi:\conbod\to\conbod_0$ was introduced by Minkowski, and became a cornerstone of the larger \textit{Brunn-Minkowski theory} along with the other concepts discussed above, such as mixed volume, Aleksandrov's variational formula, and Minkowski's inequality (among others not pertinent to our current discussion, see Schneider \cite{Sh1} and the references therein for a thorough overview). We say a function $P:\conbod\to \R^+$ is an affine invariant quantity if for every $T\in GL(n)$, $P(TK)=P(K)$. It is natural to ask if an affine invariant quantity is bounded above or below. If one uses the notation $(\Pi K)^\circ = \Pi^\circ K$, then one can check that $P(K)=\vol_n(K)^{n-1}\vol_n(\Pi^\circ K)$ is affine invariant; the inequalities of Petty \cite{CMP71} and Zhang \cite{Zhang91}, bounded this quantity from above and below respectively. Other inequalities on affine invariant quantities are the Busemann-Petty centroid inequality and the Rogers-Shephard inequality \cite{RS57}. In general, inequalities concerning projections are Shephard-type problems; the Shephard problem, asked by Shephard in 1964 \cite{She64} is the following: Consider symmetric $K,L\in\conbod_0$ such that $$h_{\Pi K}(\theta)=\vol_{n-1}(P_{\theta^{\perp}}K)\leq\vol_{n-1}(P_{\theta^{\perp}}L)=h_{\Pi L}(\theta)$$ for all $\theta\in\s^{n-1}$; is it true that $\vol_n(K)\leq \vol_n(L)?$ One sees that the above inequality means $\Pi K\subset \Pi L$. It was shown to be false in general independently by Schneider \cite{Sh2} and Petty \cite{CMP67}, but true if $L$ itself is a projection body, that is $L=\Pi M$ for some convex body $M\in\conbod.$ Therefore it is true in dimension $n=2$ for symmetric $L$, as every symmetric convex body in the plane is a projection body \cite{gardner_book}. 

Over the last several decades, various extensions of the Brunn-Minkowski theory have occurred. Each extension is a generalization of the core ingredients of the theory, which reduce to the classical case in some way. The \textit{$L^p$ Brunn-Minkowski theory}, or Firey-Brunn-Minkowski theory, was introduced in the 1960's when Firey discovered an $L^p$ extension of Minkowski addition; in the 1990's, a series of papers by Lutwak \cite{LE93,LE96} established the modern $L^p$ Brunn-Minkowski theory. Further extensions of concepts from the classical theory were extended to the $L^p$ setting; the $L^p$ analog of the projection operator was established in \cite{LYZ00}, and extended to a larger family of operators in \cite{ML05}. A family of $L^p$ extensions of Petty's inequality was established in \cite{HS2009, HS09}. The $L^p$ extension of Minkowski's existence theorem was analyzed by Chou and Wang \cite{CW06} and by Lutwak, Yang and Zhang \cite{LYZ06}. The next development of the Brunn-Minkowski theory was presented (comparatively recently) in a series of papers by Lutwak, Yang and Zhang, called the \textit{Orlicz-Brunn-Minkowski theory} \cite{HLYZ10,LYZ10_2, LYZ10} (where the first citation was done also with Haberl). In \cite{LYZ10}, the projection body operator was further generalized (and the associated Petty inequality was shown), and in \cite{HLYZ10}, the associated Minkowski's existence theorem was shown. For more on the Firey and Orlicz Brunn-Minkowski theories, the reader is encouraged to see \cite{Sh1} and the references in \cite{LYZ10}. In particular, the logarithmic Brunn-Minkowski conjecture \cite{BLYZ12,BLYZ13} is an active area of interest for many.

 In parallel with these developments, primarily over the last two decades, extensions of various problems from the classical Brunn-Minkowski setting to that of general measures were established. We say a Borel measure $\mu$ has density if it is has a locally integrable Radon-Nikodym derivative from $\R^n$ to $\R^{+}=[0,\infty)$, i.e, 
\[
\frac{d\mu(x)}{dx} = \phi(x), \text{ with } \phi \colon \R^n \to \R^+,\phi\in L^1_{\text{loc}}(\R^n);
\]
as a measure with density, the Lebesgue measure itself is denoted as $\lambda.$ Extensions of the classical Brunn-Minkowski theory to general measures often involve replacing the concept of the Lebesgue measure in a given statement with some Borel measure $\mu$ with density. We shall call the collection of such extensions the \textit{weighted Brunn-Minkowski theory}. The Lebesgue measure is a $n$-homogeneous measure, and so we must often impose homogeneity requirements on measures under consideration. We say a measure $\mu$ is $\alpha$-homogeneous, $\alpha > 0$, if $\mu(tK)=t^\alpha\mu(K)$ for all Borel measurable $K$ in the support of $\mu$ and $t>0$ such that $tK$ is in the support of $\mu$. Furthermore, if the density of $\mu$ is $\phi$, then one can can check using the Lebesgue differentiation theorem that $\phi$ is $(\alpha-n)$-homogeneous. We will need to state a broadening of a property of the Lebesgue measure.

First recall that, for $\alpha,\beta\geq 0,$ the \textit{Minkowski sum} of two Borel is given by
$$\alpha A+\beta B=\{\alpha x + \beta y: x\in A, y\in B\}.$$ We can now remind the reader that a Borel measure $\mu$ is said to be $s$-concave on a class of compact sets $\mathcal{C}$ if, for every pair $A,B \in \mathcal{C}$  and every $t \in [0,1]$, one has
	$$\mu(t A +(1-t)B)\geq \left(t\mu(A)^s +(1-t)\mu(B)^{s}\right)^{\frac{1}{s}}$$
whenever $\mu(A)\mu(B) > 0$; the right-hand side is taken to be zero otherwise. In the limit as $s\rightarrow 0$, we obtain the case of log-concavity:
	$$\mu(t A +(1-t)B)\geq\mu(A)^{t}\mu(B)^{1-t}.$$
	When $s=-\infty$, the right-hand side is taken to be $\min\{\mu(A),\mu(B)\}$. 
 The classical Brunn-Minkowski inequality states both the $1/n$-concavity and the log-concavity of the Lebesgue measure on the class of non-empty, compact sets, and this fact is used in most problems in the classical Brunn-Minkowski theory. A cornerstone of the weighted Brunn-Minkowski theory is Borell's classification of $s$-concave measures, specifically in Theorem 3.2 of \cite{Bor75}:
Let $\Omega\subset\R^n$ be an open, convex subset of $\R^d$ (possibly $\R^d$ itself), where $1\leq d \leq n$ is the dimension so that $\R^d$ is the least affine subspace containing $\Omega$. Then, $\mu$ is a Radon (locally finite and regular) $s$-concave measure on Borel subsets of $\Omega$ if, and only if, there exists $f\in L^1_{loc}(\Omega)$ such that $d\mu=fdx_{\Omega}$ and the following relations hold
	\begin{align*}
	    s\in[-\infty,0) &\longleftrightarrow f^{\frac{s}{1-sd}} \; \text{is convex}
	    \\
	    s=0 &\longleftrightarrow \log f \; \text{is concave}
	    \\
	    s\in(0,1/d) &\longleftrightarrow f^{\frac{s}{1-sd}} \; \text{is concave}
	    \\
	    s=1/d &\longleftrightarrow f \; \text{is a constant}
	    \\
	    s>1/d &\longleftrightarrow f=0 \; \text{a.e.}.
	\end{align*}
For the purposes of this paper, we will be focusing on the case when $\Omega$ is of full dimension, i.e. $d=n$ (where $\Omega$ may be $\R^n)$, and $s=p\in (0,1/n].$ To emphasize this, we will say a measure $\mu$ is $p$-concave when these two conditions are satisfied. One can see that if a $p$-concave measure $\mu$ is also $\alpha$-homogeneous, then the function given by $t\to t^{p\alpha}$ must be concave, and so $p\alpha\leq 1$. 

In \cite{HJ21, KK18,KL20, AK12,AK14,AK15,KPZv,Ro20, Zv08}, slicing problems were analyzed in the weighted setting. Of particular interest is the extension of Minkowski's existence theorem. For a convex body $K\in\conbod$ and a Borel measure $\mu$ \textit{on the boundary of $K$} with density $\phi$ the \textit{weighted surface area of $K$ with respect to $\mu$} is defined by
\begin{equation}
    \label{eq:surface_mu}
    S^{\mu}_{K}(E)=\int_{n_K^{-1}(E)}\phi(y)dy
\end{equation}
for every Borel set $E \subset \s^{n-1},$ where $dy$ represents integration with respect to the $(n-1)$-dimensional Hausdorff measure on $\partial K.$ If $K$ is of class $C^2_+$, then one has $$dS^{\mu}_{K}(u)=\phi\left(n_K^{-1}(u)\right)f_K(u)du.$$ 
The next step is to extend this definition to Borel measures $\mu$ with density. This will be done in the following way. For a Borel measure $\mu$ with density $\phi$ and a convex body $K$, denote the $\mu$-measure of the boundary of $K$ as 
\begin{equation}\mu^+(\partial K):=\liminf_{\epsilon\to 0}\frac{\mu\left(K+\epsilon B_2^n\right)-\mu(K)}{\epsilon}=\int_{\partial K}\phi(y)dy,
\label{eq_bd}
\end{equation}
where the second equality holds if there exists some canonical way to select how $\phi$ behaves on $\partial K$. A large class of functions consistent with \eqref{eq_bd} is when $\phi$ is continuous. Therefore, $S^{\mu}_{K}$ can be defined for any Borel measure $\mu$ with continuous density $\phi$ via the Riesz Representation theorem, since, for a continuous $f\in \mathcal{C}(\s^{n-1}),$
$$f\to \int_{\partial K}f(n_K(y))\phi(y)dy$$ is a linear functional. In many of our results below, all we actually require for $S^\mu_K$ to be well-defined is that $\partial K$ is in the Lebesgue set of $\phi$.

As an elucidating example: let $P$ be a polytope in $\R^n$ with $(n-1)$ dimensional facets $F_1,\dots,F_M$. Denote by $u_i\in\s^{n-1}$ the outer-unit normal to $F_i$. Suppose $\mu$ is a Borel measure on $\partial P$ with density $\varphi$ (e.g. $\mu$ is a measure on $\R^n$ with continuous density). Then,
$$dS^\mu_{P}(u)=\sum_{i=1}^M\mu(F_i)\delta_{u_i}(u), \quad \text{where }\mu(F_i)=\int_{F_i}\varphi(x)d\mathcal{H}^{n-1}(x)$$
and $\delta_{u_i}(u)$ is a Dirac point mass at $u_i$.

In \cite{GAL19}, Livshyts extended Minkowski's existence theorem to $\alpha$- homogeneous, $p$-concave measures: consider an $\alpha$-homogeneous, $\alpha>n$, Borel measure $\mu$ with density $\phi$, such that a restriction of $\phi$ on some half space is $p-$concave for $p \geq 0 .$ Let $\varphi$ be an arbitrary even measure on $\s^{n-1}$, not supported on any great hemisphere, such that $\operatorname{supp}(\varphi) \subset \operatorname{int}(\operatorname{supp}(\phi)) \cap \s^{n-1}$. Then there exists a unique, up to a set of $\mu$-measure zero, convex body $K\in\conbod$ such that
$$
dS^{\mu}_{K}(u)=d \varphi(u).
$$
The first goal of this paper is to establish further generalizations of tools from the classical Brunn-Minkowski theory to the weighted Brunn-Minkowski theory. In particular, mixed measures and a Minkowski's first inequality for mixed measures when the measure is $F$-concave, as introduced and shown in \cite{GAL19}, Aleksandrov's variational formula, and Minkowski's existence theorem:
\begin{question}
\label{q:min_exs}
Let $\mu$ be a Borel measure on $\R^n$ with locally integrable density $\phi$ and $\nu$ be a finite, even Borel measure on $\s^{n-1}$ such that $\nu$ is not concentrated on any great hemisphere. Then, does there exist a symmetric convex body $K\subset \R^n$ and a constant $c_{\mu,K}$ such that
$$d\nu(u)=c_{\mu,K}dS^{\mu}_{K}(u)?$$
\end{question}
We solve this question for a large class of Borel measures with density, denoted $\Lambda^n$ and defined as \begin{equation}
    \label{eq:lambda_prime}
    \begin{split}
    \Lambda^n=\bigg\{&\text{$\mu$ a Borel measure on $\R^n$ with continuous density such that}
    \\
    \exists \; \beta >0 &: \lim_{r\to\infty}\frac{\mu(rB_2^n)^\frac{\beta}{n}}{r} =0 \quad \text{ and } \quad \lim_{r\to 0}\frac{\mu(rB_2^n)^\frac{\beta}{n}}{r} =\infty \bigg\}.
    \end{split}
\end{equation}
Let us remark that $\Lambda^n$ is quite a rich class. For example, $\alpha$-homogeneous, $\alpha>0$, measures belong to $\Lambda^n$, as in this instance it suffices that the exponent of $r^{\frac{\alpha\beta}{n}-1}$ be negative, and so any $\beta <\frac{n}{\alpha}$ suffices. In addition, any probability measure with continuous density satisfies the first limit by taking the bound $0\leq \mu(rB_2^n)\leq 1$. Any radially decreasing measure satisfies the second limit as well, as, for $r$ small enough, if $\phi$ is the density of $\mu$, one can use that $\phi(rx)\geq \phi(x)$ and a variable substitution to obtain $\left(\mu(rB_2^n)^\frac{\beta}{n}/r\right)\geq \left(\mu(B_2^n)^\frac{\beta}{n}/r^{1-\beta}\right);$ thus any $\beta <1$ suffices. Therefore, $\Lambda^n$ contains every radially decreasing probability measure with continuous density, e.g. the Gaussian measure. Then, we prove the following:
\begin{theorem}
\label{t:third}
Let $\mu$ be a Borel measure on $\R^n$ with even, continuous density such that $\mu \in \Lambda^n$ for some $\beta>0.$ Suppose $\nu$ is a finite, even Borel measure on $\s^{n-1}$ that is not concentrated on any hemisphere.
Then, there exists a symmetric convex body $K$ such that
$$d\nu(u)=c_{\mu,K} dS^{\mu}_{K}(u); \quad c_{\mu,K}:=\mu(K)^{\frac{\beta}{n}-1}.$$
\end{theorem}
We then show in the case of $\alpha$-homogeneous measures $\mu$, $\alpha\in(0,1)\cup(1,\infty)$, the constant $c_{\mu,K}$ can be taken to be $1.$ We also analyze the equation $dS^{\mu}_{K}(u)=dS^{\mu}_{L}(u)$ and determine which conditions are required for $\mu$ so that one can conclude $K=L.$ We discuss how, for the Gaussian measure $\gamma_n$, this constant cannot be removed, but we do obtain uniqueness:
if $dS^{\gamma_n}_{K}(u)=dS^{\gamma_n}_{L}(u)$ then, if $K$ and $L$ are large enough in the sense that $\gamma_n(K),\gamma_n(L)\geq \frac{1}{2}$, one has $K=L$. This result had first been obtained by Huang, Xi and Zhao \cite{HXZ21}, who used an approach that relied on properties of the Gaussian measure.

We recall that the curvature function $f_K$ of a $C^2_+$ convex body $K$ has a useful formula: if $D^2$ denotes the spherical Hessian, then $f_K=\det (D^2 h + hI),$ where $I$ is the $(n-1)\times (n-1)$ identity matrix and $h=h_K$. Theorem~\ref{t:third} thus solves the following Monge-Amp\`ere equation: Consider a given smooth, positive, even data function $f$ on the sphere and a Borel measure $\mu\in\Lambda^n$ with even, continuous density $\phi$. Then, there exists $\beta>0$ and a symmetric, smooth convex body $K$ such that $h:=h_K$ and $c=c_{\mu,K}=\mu(K)^{\frac{\beta}{n}-1}$ solve
$$c\phi(\nabla h)\det (D^2h + h I)=f.$$

In \cite{LRZ22}, Langharst, Roysdon and Zvavitch extended Zhang's inequality and the projection operator to arbitrary Borel measures with density; the extension of the projection operator is the following:
consider a convex body $K$ and a Borel measure $\mu$ with density $\phi$ containing $\partial K$ in its Lebesgue set. Then, the weighted projection body of $K$ with respect to $\mu$ is the convex body $\Pi_{\mu} K$ whose support function is given by
\begin{equation}
    h_{\Pi_{\mu} K}(\theta) = \frac{1}{2} \int_{\partial K} |\langle \theta,n_K(y)\rangle|\phi(y) dy=\frac{1}{2}\int_{\s^{n-1}}|\langle \theta, u \rangle|dS^{\mu}_{K}(u).
\end{equation}
In \cite{KZ15}, Zvavitch and Koldobsky considered the isomorphic Busemann-Petty problem: given symmetric convex bodies $K$ and $L$ and a Borel measure $\mu$ with density such that $$\mu(K\cap \theta^\perp) \leq \mu(L\cap \theta^\perp)$$
for all $\theta\in\s^{n-1}$, does there exist some $\mathcal{L}>0$ such that $\mu(K)\leq \mathcal{L}\mu(L)?$ They showed the answer is affirmative, with the worst possible constant being $\mathcal{L}=\sqrt{n}.$ We should mention that the original Busemann-Petty problem asked this question with volume in place of $\mu$ and with $\mathcal{L}=1$ \cite{BP56}; it is true for $n=2,3,4$ and false for $n\geq 5$ \cite{Ball87,BJ91, GR94, GKS99,GKS99_2, AG90, RL75,MP92,GZ93,GZ94}. In \cite{Zv05}, Zvavitch solved the Busemann-Petty problem for general measures, obtaining the same results. 

In this paper, we will further study the projection bodies $\Pi_\mu K$, and classify them into weighted projective classes. In particular, motivated by the result of Zvavitch and Koldobsky, we answer the following \textit{isomorphic Shephard problems} for these weighted projection bodies.
 \begin{question}
\label{q:shep_meas}
Let Borel measures $\mu$ and $\nu$ have continuous densities and $K,L\in\conbod_0$ be symmetric such that
$$h_{\Pi_{\mu} K}(\theta)\leq h_{\Pi_{\nu} L}(\theta)$$ for all $\theta\in\s^{n-1}$. Does there exist a quantity $\mathcal{A}>0$ such that $\mu(K)\leq \mathcal{A}\nu(L)$?
\end{question}

The celebrated result of K. Ball \cite{Ball91_0} showed the construction of a symmetric random polytope such that, under the hypothesis that $h_{\Pi K}(\theta)\le h_{\Pi L}(\theta)$ for all $\theta\in \s^{n-1}$, and $L$ not a projection body, one had $\vol_n(K)\approx \sqrt{n}\vol_n(L).$  And yet, when $L$ is a projection body, the same hypothesis yields $\vol_n(K)\le \vol_n(L)$. Therefore, $\mathcal{A}$ may depend on $L$ and possibly the dimension. It is reasonable to suspect that $\mathcal{A}$ will depend on $\mu$ and $\nu$ as well. However, for the quantity we obtain, the dependency on $\mu$ and $\nu$ vanishes when $\mu=\nu$. We will consider Question~\ref{q:shep_meas} in the case where $\mu,\nu$ are both homogeneous of some degree; we will also discuss that the question is not well-posed for when $\mu=\nu=\gamma_n$. We list here some of the less complicated results. The constant $d_{\Pi}(L)$ below satisfies $1\leq d_{\Pi}(L)\leq \sqrt{n},$ and  $d_{\Pi}(L)=1$ when $L$ is a projection body.

\begin{theorem}[Isomorphic Shephard Problem for Homogeneous Measures]
\label{t:q_1} Suppose $\mu$ and $\nu$ are Borel measures on $\R^n$ with continuous densities, where $\nu$ is $\beta$-homogeneous, $\beta>0$ and $\mu$ is $p$-concave and $\alpha$-homogeneous, such that one of the two holds
\begin{enumerate}
    \item $\alpha \geq n$ and $p>0$
    \item $\alpha<n$ but $p=\frac{1}{\alpha}$.
\end{enumerate}
 Let $K$ and $L$ be symmetric convex bodies such that $$h_{\Pi_{\mu} K}(\theta) \le h_{\Pi_{\nu} L}(\theta)$$ for every $\theta\in\s^{n-1}.$ Then one has that
$\mu(K) \le \frac{\beta}{\alpha} \mathcal{A} d_{\Pi}(L) \nu(L)$
where
$$\mathcal{A}=\min\left\{\mathcal{A}_K,\mathcal{A}_L\right\}, \quad \mathcal{A}_K=\left(\frac{\mu(K)}{\mu(L)}\right)^{\frac{1}{\alpha}}, \quad \mathcal{A}_L=\left(\frac{\nu(L)}{\mu(L)}\frac{\beta }{\alpha}\right)^{\frac{1}{\alpha-1}}d^\frac{1}{\alpha-1}_{\Pi}(L).$$
\end{theorem}
The quantity $\mathcal{A}_K$ in the above has the advantage of being independent of $\nu$. The quantity $\mathcal{A}_L$ is in some sense a measure of asymmetry between $\mu$ and $\nu$ on $L$. In both cases, when $\mu=\nu$ one obtains
$\mu(K)\leq d^{\frac{\alpha}{\alpha -1}}_{\Pi}(L)\mu(L)$.
The case when $\alpha\in(0,n)$ and when $\mu$ is $F$-concave will be discussed as well. Besides being an extension of the classical Shephard problem to the setting of measures, there exists a geometric reason for the study of these questions. For a Borel measure $\mu$ with density, one has that \eqref{eq_bd} yields, via Fubini's theorem, that 
\begin{equation}
\label{eq:sur}
\mu^+(\partial K)=\frac{1}{\kappa_{n-1}}\int_{\s^{n-1}}h_{\Pi_{\mu} K}(\theta)d\theta.
\end{equation}
Therefore, Question~\ref{q:shep_meas} for the Gaussian measure is related to the result of Ball \cite{Ball93}:
for $K\in\conbod, \gamma^+_n(\partial K)\leq 4n^{1/4}$, which was shown to be asymptotically sharp by Nazarov \cite{Naz03}.

 This paper is structured as follows. In Section~\ref{s:min}, we extend Minkowski's existence theorem; that is, answer Question~\ref{q:min_exs} for measures in $\Lambda^n$. We also discuss uniqueness conditions, and as a special case, we obtain the result of Livshyts \cite{GAL19}. In Section~\ref{s:class}, we list the developments over the past twenty years in the application of harmonic analysis to convex geometry, following the works of Koldobsky, Ryabogin, Zvavitch, Yaskin and others \cite{AK05, KRZ04_1, KY08, KZ15}. Also, we give a classification of the projection bodies $\Pi_\mu K$ for a fixed $\mu$ from various perspectives. Section~\ref{s:shep_meas} concerns the isomorphic Shephard problems; we present various solutions to Question \ref{q:shep_meas}. Section~\ref{s:con} includes applications of the above results to the concept of capacity from potential theory. Finally, in the Appendix, we generalize Section~\ref{s:min} to the setting of the $L^p$-Brunn-Minkowski Theory.

\textbf{Acknowledgements} We would like to thank A. Zvavitch for the numerous, helpful comments and guidance. We would also like to thank Jacopo Ulivelli for the comments concerning the presentation of the definition of the weighted surface area measure.

\section{Extensions of Minkowski's Existence Theorem and Minkowski's Inequality}
\label{s:min}
In general, we will consider measures with concavity beyond Borell's classification. We say a Borel measure $\mu$ is $F$-concave, where $F$ is an invertible, (strictly) monotonic, continuous function $F$, if there exists a class $\mathcal{C}$ of Borel sets (each with finite $\mu$ measure) such that, for every $K,L\in\mathcal{C}$ and every $0<\lambda <1$ one has 
 \begin{equation}
	\label{eq:concave}
	\mu\left((1-\lambda)K+\lambda L\right)\geq F^{-1}\left((1-\lambda) F(\mu(K)) +\lambda F(\mu(L))\right).
	\end{equation}

	Our first step is understanding equality conditions in the $F$-concavity of a measure, \eqref{eq:concave}. We show that, if there is equality at a single $\lambda\in (0,1),$ then there is equality for all $\lambda\in [0,1].$
	\begin{proposition}
	\label{p:equal}
	Let $\mu$ be $F$-concave on a class of Borel sets $\mathcal{C}$ such that $F$ is monotone. Suppose we have $A,B\in\mathcal{C}$ and some $\lambda\in(0,1)$ such that
	\begin{equation*}
	\mu\left((1-\lambda)K+\lambda L\right) = F^{-1}\left((1-\lambda) F(\mu(K)) +\lambda F(\mu(L))\right)
	\end{equation*}
	Then, equality holds for all $\lambda \in [0,1]$.
	\end{proposition}
	\begin{proof}
	Since $F$ is monotone and invertible, it is either increasing or decreasing, and so the function given by $f(t)=F\left(\mu\left((1-t)K+t L\right)\right)$
	is either concave or convex. Denote the linear function $g(t)=(1-t) F(\mu(K)) +t F(\mu(L))$. We have that $g(0)=f(0)=\mu(K)$ and $g(1)=f(1)=\mu(L)$. By hypothesis however, we also have $f(\lambda)=g(\lambda)$. But, since either $f$ or $-f$ is concave, this implies $f=g$ on all of $[0,1].$
	\end{proof}
	
	In the case of volume and the Brunn-Minkowski inequality, equality conditions imply that $K$ is homothetic to $L$ \cite[Chapter 7]{Sh1}. When $F(x)=x^s$, $s\in[-\infty,1/n]$ one can use the equality conditions of the Pr{\'e}kopa-Leindler inequality; this has been done previously by E. Milman and Rotem in \cite{MR14}, and we outline this in the following proposition.
	
	\begin{proposition}[Milman and Rotem \cite{MR14}]
	\label{p:mr}
	Suppose $\mu$ is $s$-concave, $s\in [-\infty,1/n]$, on Borel subsets of $\R^n$. If $A$ and $B$ are Borel sets such that $\mu(A)\mu(B) >0$ and
		$$\mu(t A +(1-t)B) = \left(t\mu(A)^s +(1-t)\mu(B)^{s}\right)^{\frac{1}{s}},$$
	then there exists $\alpha >0$ and $\beta\in \R^n$ such that $A=\alpha B + \beta$. Additionally, there exists $c>0$ such that $\varphi(\alpha x+\beta)=c\varphi(x)$ for almost every $x\in B$, where $\varphi$ is the density of $\mu$.
	\end{proposition}

	One sees very quickly that an all encompassing schema for general $\mu$ and $F$ cannot be expected in general. In fact, the equality conditions of the Ehrhard inequality constituted an entire paper \cite{EHR2}. Recall that the Gaussian measure is given by $$d\gamma_n(x)=\frac{1}{(2\pi)^{\frac{n}{2}}}  e^{-\frac{|x|^2}{2}} dx.$$  The Ehrhard inequality states that, for $0<t<1$ and closed convex sets $K$ and $L$ in $\R^{n}$,
\begin{equation}\label{e:Ehrhard_ineq}
\Phi^{-1}\left(\gamma_{n}((1-t) K+tL)\right) \geq(1-t) \Phi^{-1}\left(\gamma_{n}(K)\right)+t \Phi^{-1}\left(\gamma_{n}(L)\right),
\end{equation}
i.e. $\Phi^{-1}\circ\gamma_n$ is concave, where $\Phi(x)=\gamma_{1}((-\infty, x))$. It was first proven by Ehrhard himself \cite{EHR1,EHR2}.  Lata\l{}a \cite{Lat96} generalized Ehrhard’s result to the case of arbitrary Borel $K$ and convex $L$; the general case of Ehrhard's inequality on Borel sets was proven by Borell \cite{Bor03}. Equality is obtained when $\gamma_{n}(K) \gamma_{n}(L)>0$ in the convex case if, and only if, either $K=\R^{n}, L=\R^{n}, K=L$, or both $K$ and $L$ are half-spaces, one contained in the other. So, if $K,L\in\conbod$, one must have $K=L.$ But also, since $\Phi$ is strictly log-concave, so too is the Gaussian measure. That is, Equation \ref{e:Ehrhard_ineq} also yields for $0<t<1$ and closed convex sets $K$ and $L$ in $\R^{n}$,
\begin{equation}\label{e:Log_gamma_concave}
\gamma_{n}((1-t) K+t L) \geq \gamma_{n}(K)^{1-t} \gamma_{n}(L)^{t}.
\end{equation}
If we have equality, then $K=L.$ 
	
	We now proceed with the extension of mixed volume to arbitrary measures, which is as follows: given $K,L\in\conbod$ and a Borel measure $\mu$ with density, one has
	\begin{equation}
	    \mu(K,L)=\liminf_{\epsilon\to0}\frac{\mu(K+\epsilon L)-\mu(K)}{\epsilon},
	    \label{eq:arb_mixed_0}
	\end{equation}
(see, for example, \cite{FLMZ23_1,FLMZ23_2,HJ21,LRZ22, GAL19, MR14} for recent uses of this definition). From the definition, one has that $\mu(K,tL)=t\mu(K,L)$, and, if $\mu$ is $\alpha$-homogeneous, that is for $t>0,$
	$\mu(tA)=t^\alpha\mu(A)$ for Borel sets, one then has $\mu(tK,L)=t^{\alpha-1}\mu(K,L).$
	Also, we see that, for $K\in\conbod$,
	$$\mu(tK,K)=\liminf_{\epsilon\to0}\frac{\mu((t+\epsilon)K)-\mu(tK)}{\epsilon}=\diff{\mu(tK)}{t};$$
	where the limit infiumum is equal to the limit since $q(t)=\mu(tK)$ is a non-decreasing function in $t.$ Since $q(0)=0,$ one then has
	\begin{equation}\mu(K)=\int_{0}^1\mu(tK,K)dt. \label{eq:int_meas}
	\end{equation}
	From the definition, one can extend Equation~\ref{eq:min_ineq}, Minkowski's inequality, to arbitrary measures, as shown in \cite{GAL19}. We replicate the proof here for completeness.
\begin{theorem}[Minkowski's Inequality for $F$-Concave Measures,\cite{GAL19}] 
	\label{t:min_eq}
	Let $\mu$ be a Borel measure on $\R^{n}$, such that $\mu$ is $F$-concave, $F$ is differentiable, with respect to a class of Borel sets $\mathcal{C}$. Then, for every $K,L\in\mathcal{C}$, one has that:

$$
\mu(K, L) \geq \mu(K, K)+\frac{F(\mu(L))-F(\mu(K))}{F^{\prime}(\mu(K))}.
$$
\end{theorem}
\begin{proof}
In Equation~\ref{eq:concave}, instead of $K$ consider $K_\lambda=\frac{K}{1-\lambda}$. Then we have that $\mu(K+\lambda L)=\mu((1-\lambda)K_\lambda+\lambda L) \geq F^{-1}\left((1-\lambda) F(\mu(K_\lambda)) +\lambda F(\mu(L))\right)$. Notice that we obtain equality when $\lambda=0$, primarily we obtain $\mu(K)$. We must then have that the far left grows faster than the far right at $\lambda=0$; from the chain rule, product rule and inverse function theory, we obtain
\begin{align*}\mu(K,L)&=\liminf_{\lambda \to 0}\frac{\mu(K+\lambda L)-\mu(K)}{\lambda}
\\
&\geq \frac{F(\mu(L))-F(\mu(K))}{F^{\prime}(\mu(K))}+\frac{1}{F^{\prime}(\mu(K))}\left((1-\lambda)\diff{}{\lambda}\left[F(\mu(K_\lambda))\right]\right)\bigg|_{\lambda=0}.\end{align*}
From the fact that $\diff{}{\lambda}\left[F(\mu(K_\lambda))\right]\bigg|_{\lambda =0}=F^{\prime}(\mu(K))\diff{\mu(K_\lambda)}{\lambda}\bigg|_{\lambda=0}$, we have our claim, since, from the chain rule, $\diff{\mu(K_\lambda)}{\lambda}\bigg|_{\lambda=0}=\mu(K,K)$.

\end{proof}
	We will establish equality conditions of Theorem~\ref{t:min_eq} below, after establishing an extension of Aleksandrov's variational formula to arbitrary measures. This will be done using variational techniques following \cite{HLYZ10}. One has that, up to a set of spherical Lebesgue measures zero, the map $r_K:\s^{n-1}\to \partial K$ given by $r_K(u)=\rho_K(u)u$  is well-defined. For future use, we shall compute the Jacobian of the coordinate transformation $r_K$; first recall the following proposition from \cite{HLYZ16}.
	\begin{proposition}[Lemma 2.9 in \cite{HLYZ16}]
	Fix a convex body $K\in\conbod$. Suppose $f:\s^{n-1}\to \R^+$. Then,
	\begin{equation}
	    \label{eq:cone}
	    \int_{\s^{n-1}}f(u)\rho^n_K(u)du=\int_{\partial^\prime K}\langle x,n_K(x)\rangle f\left(\frac{x}{|x|}\right)dx.
	\end{equation}
	\end{proposition}
	
	\begin{proposition}
	\label{p:jacob}	For $K\in\conbod_0$, the Jacobian of $r_K:\s^{n-1} \to \partial K$ is $\frac{\rho^n_K(u)}{h_K(n_K(r_K(u)))}$ up to a set of measure zero.
	\end{proposition}
	\begin{proof}
	For non-negative $F\in L^1(\partial K)$, we set $f(u)=\frac{F(r_K(u))}{h_K(n_K(r_K(u)))}$ in Equation~\ref{eq:cone}. Then, one obtains
	\begin{align*}\int_{\s^{n-1}}F(r_K(u))&\frac{\rho^n_K(u)}{h_K(n_K(r_K(u)))}du
 \\
 &=\int_{\partial^\prime K}\langle x,n_K(x)\rangle \frac{F\left(r_K\left(\frac{x}{|x|}\right)\right)}{h_K\left(n_K\left(r_K\left(\frac{x}{|x|}\right)\right)\right)}dx.\end{align*}
	Where the integral on the left-hand side is understood to be over $$\s^{n-1}\setminus\{u\in\s^{n-1}: n_K(r_K(u)) \text{ is not well-defined}\}.$$
	Notice on $\partial K$ that $r_K\left(\frac{x}{|x|}\right)=\rho_K\left(\frac{x}{|x|}\right)\frac{x}{|x|}=\rho_K(x)x=x.$ Furthermore, one has $\langle x,n_K(x)\rangle=h_K(n_K(x)).$ Hence,
	$$\int_{\s^{n-1}}F(r_K(u))\frac{\rho^n_K(u)}{h_K(n_K(r_K(u)))}du=\int_{\partial^\prime K}F(x)dx=\int_{\partial K}F(x)dx.$$
	\end{proof}
	
	In \cite{HLYZ16}, the concept of a logarithmic family of Wulff shapes was introduced: for $h,f\in C(\s^{n-1})$ and some small $\delta$, define a function $h_t:\s^{n-1}\to (0,\infty)$ via
	\begin{equation}
	    \log(h_t(u))=\log(h(u)) +tf(u) + o(t,u)
	    \label{eq:log_fam}
	\end{equation}
	where $o(t,u)/t \rightarrow 0$ as $t\to 0$ for all $u\in\s^{n-1}.$ Then, the Wulff shapes $[h_t]$ are said to be the logarithmic family of Wulff shapes formed by the pair $(h,f).$ It was shown in \cite[Lemma 4.3]{HLYZ16} that the radial functions of a logarithmic family satisfy, for almost all $u\in\s^{n-1}$, the following limit:
	$$\diff{\log\rho_{[h_t]}(u)}{t}\bigg|_{t=0}=\lim_{t\to 0}\frac{\log\rho_{[h_t]}(u)-\log\rho_{[h]}(u)}{t}=f(n_{[h]}(r_{[h]}(u))),$$
	and, furthermore, one has that there exists some $M>0$ such that $$|\log\rho_{[h_t]}(u)-\log\rho_{[h]}(u)|<M|t|$$ for every $|t| <\delta.$
	Using this result, we obtain the following lemma.
	
	\begin{lemma}[First Variation of Radial Functions of Wulff Shapes]
	\label{l:first}
	Consider a convex body $K\in\conbod_0$ and $f\in C(\s^{n-1})$. Then, there exists a small $\delta >0$ and $M>0$ such that, for $|t|< \delta $:
 \begin{enumerate}
     \item $h_t(u)=h_K(u)+tf(u)$ is positive for all $u\in\s^{n-1}$ and
     \item $|\rho_{[h_t]}(u)-\rho_{K}(u)|<M|t|$ for almost all $u\in\s^{n-1}$.
 \end{enumerate} Additionally, one has, for almost all $u\in\s^{n-1}$ up to a set of spherical Lebesgue measure zero, that
	$$\diff{\rho_{[h_t]}(u)}{t}\bigg|_{t=0}=\lim_{t\to 0}\frac{\rho_{[h_t]}(u)-\rho_{K}(u)}{t}=\frac{f(n_{K}(r_{K}(u)))}{h_K(n_K(r_K(u)))}\rho_K(u).$$
	\end{lemma}
	\begin{proof} Let $\delta_1$ be small enough so that $h_t(u)=h_K(u)+tf(u)$ is positive for all $u\in\s^{n-1}$ and $|t|< \delta_1$, which exists since $h_K$ is positive and $f$ is bounded on $\s^{n-1}$.	We have that, for a fixed $u$, $h_t$ is a linear function in $t$, positive in some neighborhood of zero. Thus, we can take the Taylor series expansion of $\log(h_t)$: $$\log(h_t)=\log(h_K)+\frac{tf}{h_K}+o(t).$$ Therefore, we see $h_t$ is a logarithmic family of the pair $(h_K,\frac{f}{h_K})$. Consequently, from \cite[Lemma 4.3]{HLYZ16}, we have for almost all $u\in\s^{n-1}$ that
	$$\diff{\log\rho_{[h_t]}(u)}{t}\bigg|_{t=0}=\frac{f(n_{K}(r_{K}(u)))}{h_K(n_K(r_K(u)))}.$$
	However, from the chain rule, we also have $$\diff{\log\rho_{[h_t]}(u)}{t}\bigg|_{t=0}=\rho^{-1}_{[h_t]}(u)\bigg|_{t=0}\diff{\rho_{[h_t]}(u)}{t}\bigg|_{t=0}.$$
	Solving for $\diff{\rho_{[h_t]}(u)}{t}\bigg|_{t=0}$ yields the variational result. This implies that for every $\epsilon>0$, there exists $\delta_2=\delta_2(\epsilon)>0$ such that, when $|t|<\delta_2:$
	$$\bigg|\bigg|\frac{\rho_{[h_t]}(u)-\rho_{K}(u)}{t}\bigg|-\bigg|\diff{\rho_{[h_t]}(u)}{t}\bigg|_{t=0}\bigg|\leq\bigg|\frac{\rho_{[h_t]}(u)-\rho_{K}(u)}{t}- \diff{\rho_{[h_t]}(u)}{t} \bigg|_{t=0}\bigg|\leq\epsilon.$$
	Consequently, we have
	$$\bigg|\frac{\rho_{[h_t]}(u)-\rho_{K}(u)}{t}\bigg|\leq\bigg|\diff{\rho_{[h_t]}(u)}{t} \bigg|+\epsilon=\rho_K(u)\bigg|\diff{\log\rho_{[h_t]}(u)}{t}\bigg|_{t=0}\bigg| +\epsilon. $$
	But, from Lemma 4.3 of \cite{HLYZ16}, we have that there exists a constant $A>0$ such that $$|\log\rho_{[h_t]}(u)-\log\rho_{K}(u)|<A|t|.$$ Dividing through by $t$ and taking the limit as $t\to 0$ then yields $$\bigg|\diff{\log\rho_{[h_t]}(u)}{t}\bigg|_{t=0}\bigg|<A.$$ Thus, fixing a small $\epsilon$ and setting $M=A \underset{u\in\s^{n-1}}{\max} \rho_K(u)+\epsilon$ yields the second claim, by setting $\delta=\min\{\delta_1,\delta_2\}$.
	\end{proof}
	We would like to mention that the first variation of support functions of Wulff shapes was done by Saroglou, Soprunov and Zvavitch in \cite{SSZ19}, where they obtained
	$$\diff{h_{[h_t]}(u)}{t}\bigg|_{t=0}=\lim_{t\to 0}\frac{h_{[h_t]}(u)-h_{K}(u)}{t}=f(u).$$

	\begin{lemma}[Aleksandrov's Variational Formula For Arbitrary Measures]
	\label{l:second}
	Let $\mu$ be a Borel measure on $\R^n$ with locally integrable density $\phi$. Let $K$ be a convex body containing the origin in its interior, such that $\partial K$, up to set of $(n-1)$-dimensional Hausdorff measure zero, is in the Lebesgue set of $\phi$. Then, for a continuous function $f$ on $\s^{n-1}$, one has that
	$$\lim_{t\rightarrow 0}\frac{\mu([h_K+tf])-\mu(K)}{t}=\int_{\s^{n-1}}f(u)dS^{\mu}_{K}(u).$$
	\end{lemma}
	\begin{proof}
	Begin by writing in polar coordinates, with $h_t=h_K+tf$:
	$$\mu([h_t])=\int_{\s^{n-1}}\int_0^{\rho_{[h_t]}(u)}\phi(ru)r^{n-1}drdu=\int_{\s^{n-1}}F_t(u)du,$$
	where
	$F_t(u)=\int_0^{\rho_{[h_t]}(u)}\phi(ru)r^{n-1}dr.$ First observe that $\diff{F_t}{t}\bigg|_{t=0}=F^\prime_0$ exists by the Lebesgue differentiation theorem; by denoting $I=[\rho_{K}(u),\rho_{[h_t]}(u)],$ we have
	\begin{align*}\lim_{t\to 0}\frac{F_t(u)-F_0(u)}{t}&=\lim_{t\to 0}\frac{1}{t}\int_{I}\phi(ru)r^{n-1}dr
 \\
	&=\lim_{t\to 0}\frac{\rho_{[h_t]}(u)-\rho_{K}(u)}{t}\frac{1}{|I|}\int_I \phi(ru)r^{n-1}dr
	\\
	&=\phi(r_K(u))\frac{f(n_{K}(r_{K}(u)))}{h_K(n_K(r_K(u)))}\rho^{n}_K(u)\end{align*} for almost all $u\in\s^{n-1}$ via Lemma~\ref{l:first}; note $\phi(r_K(u))$ exists from the assumption on $\phi$. But also, since
	$$\frac{f(n_{K}(r_{K}(u)))}{h_K(n_K(r_K(u)))}\rho_K(u)=\lim_{t\to 0}\frac{\rho_{[h_t]}(u)-\rho_{K}(u)}{t}\leq M$$
	for some $M>0$, via Lemma~\ref{l:first}, we have that the derivative is dominated by the integrable function $M\left(\underset{u\in\s^{n-1}}{\max}\rho_K(u)\right)^{n-1}\phi(r_K(u))$. Thus, our claim follows from using dominated convergence to differentiate underneath the integral sign:
	\begin{align*}
	    \int_{\s^{n-1}}F^\prime_0(u)du&=\int_{\s^{n-1}}\phi(r_K(u))\frac{f(n_{K}(r_{K}(u)))}{h_K(n_K(r_K(u)))}\rho^n_K(u)du
	    \\
	    &=\int_{\partial^\prime K}\phi(y)f(n_K(y))dy = \int_{\s^{n-1}}f(u)dS^{\mu}_{K}(u),
	\end{align*}
	where the second equality follows from Proposition~\ref{p:jacob} and the third equality follows from the Gauss map. 
	\end{proof}

  \begin{remark}
  \label{r:second}
     It was suggested to us by Ulivelli that in Lemma~\ref{l:second}, one should be able to drop the assumption that $K$ contains the origin in its interior. Indeed, this is the case:
     for $K\in\conbod$, let $K-\int_{K}xdx = K^\prime \in \conbod_0.$ For a Borel measure $\mu$ with density $\phi$ containing $\partial K$ in its Lebesgue set, let $\mu^\prime$ be the Borel measure with density $\phi^\prime(y) = \phi(y+\int_K x dx)$. The claim then follows since, $dS^{\mu^\prime}_{K^\prime}(u)=dS^\mu_K(u)$, $\mu(K)=\mu^\prime(K^\prime)$, and, for t small enough, $[h_{K^\prime}+tf]=[h_K+tf]-\int_K x dx$ implies $\mu^\prime([h_{K^\prime}+tf])=\mu([h_K+tf]).$
 \end{remark}

 \noindent We recall that $L^1(\R^n)$ is partitioned into equivalence classes, where $\phi_1$ and $\phi_2$ are equivalent if they differ on a set of measure zero. Let $\mu_1$ be a measure with density $\phi_1$ containing $\partial K$ in its Lebesgue set for some $K\in\conbod$ (up to a set of $(n-1)$ Hausdorff measure zero). Then, let $\mu_2$ be the measure with density $\phi_2$, where $\phi_2$ is in the same equivalence class as $\phi_1$, but does not necessarily have $\partial K$ in its Lebesgue set. Then, the proof of Lemma~\ref{l:second} still works for the measure $\mu_2$. However, one will not necessarily obtain integration of $\phi_2$ over $\partial K$ from the use of the Lebesgue differentiation theorem, instead they will obtain $\phi_1$ (a.e. on $\partial K$). Thus, we can \textit{define} the weighted surface area measure of $K$ with respect to $\mu_2$, the Borel measure $S^{\mu_2}_K$ on $\s^{n-1}$, to merely \textit{be} the Borel measure $S^{\mu_1}_K$.
	
	\begin{theorem}[Representation of Mixed Measures]
	Let $K$ be a convex body in $\R^n$ and $L$ be a compact, convex set in $\R^n$. Suppose $\mu$ is a Borel measure on $\R^n$ with density $\phi$ containing $K$ in its support and $\partial K$ in its Lebesgue set. Then, one has that
	\begin{equation}
	    \mu(K,L)=\int_{\s^{n-1}}h_L(u)dS^{\mu}_{K}(u)=\int_{\s^{n-1}}h_L(u)\phi(n_K^{-1}(u))f_K(u)du,
	    \label{eq:arb_mixed}
	\end{equation}
	where the second equality holds if $K$ is of class $C^2_+$.
	\end{theorem}
	\begin{proof}
	In Remark~\ref{r:second}, set $f=h_L$ for $L$ a compact, convex set.
	\end{proof}
    In \cite{MR15}, Equation~\ref{eq:arb_mixed} was obtained in the case of the Gaussian measure. In \cite{GAL19}, Equation~\ref{eq:arb_mixed} was obtained for measures with continuous densities using a different approach, what is known as the second Nazarov coordinate system. The advantage of the approach done here is that we were able to establish Lemma~\ref{l:second} first. In fact, in \cite{HKL21}, the second Nazarov coordinate system was also used to prove Lemma~\ref{l:second} in the case of measures with continuous densities and origin symmetric $K\in\conbod_0$. Kolesnikov and Milman have previously shown the case of Lemma~\ref{l:second} for log-concave measures \cite{KM18}. In the field of geometric measure theory, many variants of Lemma~\ref{l:second} have been shown for the study of curvature flow on manifolds under various assumptions (e.g. smoothly embedded); see \cite{CM12} and the references therein.
    
    Henceforth, we will always assume that the function $\phi$ in Lemma~\ref{l:second} is continuous, and thus $\partial K$ is in the Lebesgue set of $\phi$ for every $K\in\conbod$. However, most of the results listed here actually hold locally, i.e. when $\phi$ is not necessarily continuous, but $\partial K$ is in the Lebesgue set of $\phi$.
	
	\begin{remark}
	Let $K\in\conbod$ be of class $C^2_+$ and let $\mu$ be a Borel measure on $\R^n$ with continuous density $\phi$. Then, from Equations \ref{eq:int_meas} and \ref{eq:arb_mixed} 
	we can write
$$\mu(K)=\int_{0}^1\mu(tK,K)dt=\int_0^1\int_{\s^{n-1}}h_K(u)\phi(tn^{-1}_K(u))dS_{tK}(u)$$
and use that, in this instance, $dS_{tK}(u)=t^{n-1}f_K(u)du$ along with Fubini's theorem to recover the representation from \cite{CLM17,KL21,GAL19}
\begin{equation}
\mu(K)=\int_{\s^{n-1}}h_K(u)f_K(u)\int_{0}^1 t^{n-1}\phi(tn^{-1}_K(u))dtdu.
\label{eq:meas_int}
\end{equation}
\end{remark}

We are now able to establish equality conditions in Minkowski's inequality, which we will see is connected to the equality conditions of Equation~\ref{eq:concave}.
	\begin{corollary}[Minkowski's Inequality - Concavity link]
\label{cor:minkowski}
Let $K,L\in\conbod$ and suppose a Borel measure $\mu$ with continuous density is such that $\mu$ is $F$-concave and $F$ is differentiable. Then, $$
\mu(K, L) = \mu(K, K)+\frac{F(\mu(L))-F(\mu(K))}{F^{\prime}(\mu(K))},
$$
if, and only if, there is equality in Equation~\ref{eq:concave}.
\end{corollary}
\begin{proof}
Suppose there is equality in Equation~\ref{eq:concave}. Then, one has
$$F \circ \mu((1-\lambda) K+\lambda L)-(1-\lambda) F \circ \mu(K)-\lambda F \circ \mu(L)=0, \; \text{for }\lambda\in[0,1].$$
Differentiating in $\lambda$ yields $$\diff{}{\lambda} F \circ \mu((1-\lambda) K+\lambda L)=F\circ\mu(L)-F\circ\mu(K).$$
Using the chain rule and evaluating at $\lambda=0$ yields
$$F^{\prime}(\mu(K))\diff{}{\lambda} \mu\left(\left[\lambda h_{L}+(1-\lambda) h_{K}\right]\right)\bigg|_{\lambda=0}=F\circ\mu(L)-F\circ\mu(K).$$
Where we used that $(1-\lambda)K +\lambda L$ is the Wulff shape of the function $\lambda h_L + (1-\lambda)h_K$  (if need be, shift both $\mu$ and $K$ so that $K\in\conbod_0$).
But notice that 
$$\diff{}{\lambda} \mu\left(\left[\lambda h_{L}+(1-\lambda) h_{K}\right]\right)\bigg|_{\lambda=0}=\int_{\s^{n-1}}(h_L(u)-h_K(u))dS^{\mu}_{K}(u),$$
from Remark~\ref{r:second}, with $f=h_L-h_K$. Since we have that $\int_{\s^{n-1}}(h_L(u)-h_K(u))dS^{\mu}_{K}(u)=\mu(K,L)-\mu(K,K)$ from Equation~\ref{eq:arb_mixed}, we have equality in Minkowski's inequality, Theorem~\ref{t:min_eq}. Conversely, suppose we have equality in Theorem~\ref{t:min_eq}. Then, by running the argument backwards, this implies that
$$\diff{}{\lambda} F \circ \mu((1-\lambda) K+\lambda L)\bigg|_{\lambda=0}=F\circ\mu(L)-F\circ\mu(K).$$
Next, let $f(t)=F\left(\mu\left((1-t)K+t L\right)\right)$, which is either concave or convex by hypothesis, and denote the linear function $g(t)=(1-t) F(\mu(K)) +t F(\mu(L))$. We have that $f(0)=g(0)$ and $f(1)=g(1)$. However, we have also shown that $f^\prime(0)=g^\prime(0)$. From the concavity of $f$ or $-f$, it follows that $f=g$ on $[0,1]$.
\end{proof}


We now work to establish Minkowski's existence theorem, that is resolve Question~\ref{q:min_exs} for a certain class of measures. We shall first establish uniqueness.
\begin{proposition}
\label{p:uni_0}
Let $\mu$, a Borel measure with continuous density, be $F$-concave such that $F$ is differentiable. Suppose $K,L\in\conbod$ such that $$dS^{\mu}_{K}=dS^{\mu}_{L}.$$ Then
$$\frac{F(\mu(L))-F(\mu(K))}{F^{\prime}(\mu(K))}\leq \frac{F(\mu(L))-F(\mu(K))}{F^{\prime}(\mu(L))}.$$
\end{proposition}
\begin{proof}
From Theorem~\ref{t:min_eq}, one has that 
$$
\mu(K, L) - \mu(K, K)\geq \frac{F(\mu(L))-F(\mu(K))}{F^{\prime}(\mu(K))}.
$$
From the symmetry of the expression, we also have
$$
\mu(L, K) - \mu(L, L) \geq \frac{F(\mu(K))-F(\mu(L))}{F^{\prime}(\mu(L))}.
$$
Inserting the integral representation from Equation~\ref{eq:arb_mixed}, we then have
$$\int_{\s^{n-1}}(h_L(u)-h_K(u))dS^{\mu}_{K}(u)\geq \frac{F(\mu(L))-F(\mu(K))}{F^{\prime}(\mu(K))}$$
and
$$
\int_{\s^{n-1}}(h_K(u)-h_L(u))dS^{\mu}_{L}(u) \geq \frac{F(\mu(K))-F(\mu(L))}{F^{\prime}(\mu(L))}.
$$
Multiplying the second expression by $-1$ and using that $dS^{\mu}_{K}(u)=dS^{\mu}_{L}(u)$ by hypothesis yields the result.
\end{proof}
\begin{remark}
\label{r:uni_1}
If $\mu(K)=\mu(L)$ and $dS^{\mu}_{K}=dS^{\mu}_{L}$, we obtain equality in Minkowski's inequality. Thus, by Corollary~\ref{cor:minkowski}, we have equality in Equation~\ref{eq:concave}.
\end{remark}
We next provide an example of dropping the assumption $\mu(K)=\mu(L)$ by instead putting restriction on the function $F$.

\begin{lemma}
\label{l:uni}
Let $\mu$, a Borel measure with continuous density, be $F$-concave such that $F$ is differentiable and there exists some $a\geq 0$ such that, for every $b,c\in [a,\mu(\R^n))$ one has $(F(b)-F(c))(F^\prime(b)-F^\prime(c))\geq 0$, $F^\prime(b)F^\prime(c)>0$, and $F^\prime(b)\neq F^\prime(c)$. Suppose $K,L\in\conbod$ such that $$dS^{\mu}_{K}=dS^{\mu}_{L}.$$ If $\mu(K),\mu(L)\geq a$, then $\mu(K)=\mu(L)$ and there is equality in Equation~\ref{eq:concave}.
\end{lemma}
\begin{proof}
From Proposition~\ref{p:uni_0} we can write
$$\frac{F(\mu(L))-F(\mu(K))}{F^{\prime}(\mu(K))}\leq \frac{F(\mu(L))-F(\mu(K))}{F^{\prime}(\mu(L))};$$
this yields upon re-arrangement that $$\left(F(\mu(L))-F(\mu(K))\right)\left(F^\prime(\mu(L))-F^\prime(\mu(K))\right) \leq 0,$$ since \\ $F^\prime(\mu(K))F^\prime(\mu(L))>0.$ However, by hypothesis, we also have\\ $\left(F(\mu(L))-F(\mu(K))\right)\left(F^\prime(\mu(L))-F^\prime(\mu(K))\right) \geq 0.$ Hence, we have equality, and so $F(\mu(K))=F(\mu(L))$ which means $\mu(K)=\mu(L)$. Furthermore, equality is obtained in Minkowski's inequality, and so, by Corollary~\ref{cor:minkowski}, we then have equality in Equation~\ref{eq:concave}.
\end{proof}
\begin{example}
\label{ex:minkowski}
In Lemma~\ref{l:uni}, the following satisfy the requirements of $F$:
\begin{enumerate}
    \item F is an increasing, strictly convex function on $[a,\mu(\R^n)).$
    \item F is a decreasing, strictly concave function on $[a,\mu(\R^n)).$
\end{enumerate}
\end{example}

Using Lemma~\ref{l:uni}, we immediately obtain a result for the standard Gaussian measure on $\R^n$. Notice that the function $\Phi^{-1}(x)$ from Ehrhard's inequality, Equation~\ref{e:Ehrhard_ineq}, is strictly increasing on $[0,1]$, as $$\diff{\Phi^{-1}(x)}{x}=\sqrt{2\pi}e^{\frac{(\Phi^{-1})^2(x)}{2}}.$$ Furthermore, one has $\diff[2]{\Phi^{-1}(x)}{x}=2\pi\Phi^{-1}(x)e^{(\Phi^{-1})^2(x)}$; therefore $(\Phi^{-1})^\prime(x)$ is strictly increasing on $[1/2,1]$, that is $\Phi^{-1}(x)$ is convex on $[1/2,1]$. By Item 1 in Example~\ref{ex:minkowski}, the hypotheses of Lemma~\ref{l:uni} are satisfied. Notice the equality conditions of Equation~\ref{e:Ehrhard_ineq} are $K=L$. Thus, we obtain the result of Huang, Xi and Zhao \cite{HXZ21} from Lemma~\ref{l:uni}, with $F(x)=\Phi^{-1}(x)$, the Ehrhard function.
\begin{theorem}
\label{t:ehr}
Suppose $K,L\in\conbod$ such that $\gamma_n(K),\gamma_n(L)\geq 1/2$ and $$dS^{\gamma_n}_{K}=dS^{\gamma_n}_{L}.$$ Then, one has $K=L$.
\end{theorem}

We now wish to state uniqueness for $p$-concave, $\alpha$-homogeneous measures, $\alpha>0$. However, unfortunately the function $F(x)=x^p$ is increasing for all $p>0$ but $F^\prime(x)$ is decreasing for $p<1$, and therefore for such $p$ we cannot appeal to Lemma~\ref{l:uni}.

\begin{proposition}
\label{p:min_exist}
Let a Radon measure $\mu$ be $\alpha$-homogeneous, $\alpha>0,$ and $p$-concave. Suppose $K,L\in\conbod$ satisfy
$$dS^{\mu}_{K}(u)=dS^{\mu}_{L}(u).$$
Then, $$\mu(K)-\mu(L)\geq \frac{\mu^p(K)-\mu^p(L)}{\alpha p\mu^{p-1}(L)}.$$
\end{proposition}
\begin{proof}
From Equation~\ref{eq:int_meas}, we can write \begin{align*}\alpha\mu(K)&=\alpha\int_{0}^1\mu(tK,K)dt= \mu(K,K)=\int_{\s^{n-1}}h_K(u)dS^{\mu}_{K}(u)
\\
&= \int_{\s^{n-1}}h_K(u)dS^{\mu}_{L}(u)=\mu(L,K).\end{align*}
The result is immediate from Minkowski's inequality, Theorem~\ref{t:min_eq}.
\end{proof}
We now provide an instance where we obtain that the two convex bodies must coincide.
\begin{lemma}
\label{l:min_exist}
Let $\mu$ be a Radon measure that is $\alpha$-homogeneous, $\alpha>0,$ and $p=1/\alpha$-concave. Suppose $K,L\in\conbod$ satisfy
$$dS^{\mu}_{K}(u)=dS^{\mu}_{L}(u).$$
Then, $\mu(K)=\mu(L)$ and $K=aL +b$ $\mu$-almost everywhere for some $a\in\R, b\in\R^n$.
\end{lemma}
\begin{proof}
By Proposition~\ref{p:min_exist}, with $p\alpha=1$, we have from the symmetry of the result that
$$\left(\frac{\mu(K)}{\mu(L)}\right)^{1-p}\geq 1 \quad \text{ and } \quad 1\geq \left(\frac{\mu(K)}{\mu(L)}\right)^{1-p},$$
so $\mu(K)=\mu(L)$ and we have equality in Minkowski's inequality, Theorem~\ref{t:min_eq}. Thus, by Corollary~\ref{cor:minkowski} there is equality in Equation~\ref{eq:concave}. By Proposition~\ref{p:mr}, $K=aL + b$ for some $a>0$ and $b\in\R^n$.
\end{proof}
\begin{remark}
\label{r:min_exist}
In Lemma~\ref{l:min_exist}, if $K$ and $L$ are centered, then $b=0$, where a convex body is centered if the vector $\int_K x dx$ is the origin. Since Lemma~\ref{l:min_exist} concludes with $\mu(K)=\mu(L)$, we also have $a=1$. Thus, if $K$ and $L$ are centered, then one has $K=L$. \end{remark}

In general, for $p\alpha < 1$, the conclusion of Lemma~\ref{l:min_exist} fails to hold. However, suppose $\mu$ is an $\alpha$-homogeneous measure such that $\alpha\geq n$. Then, we have that $\mu$ being $\alpha$-homogeneous implies $\phi$, the density of $\mu$, is $r\!=\! \alpha\!-\!n$$\geq \!0$ homogeneous. Furthermore, if $\mu$ is $p$-concave, then from the Borell classification of concave functions \cite{Bor75}, one has $\phi$ is $s=\frac{p}{1-pn}$ concave. Since $p\leq \frac{1}{\alpha}\leq \frac{1}{n}$, $\phi$ has a positive concavity. The following proposition from \cite{GAL19} will be of great use: \\
\begin{proposition}[Proposition A.2. in \cite{GAL19}]\label{p:gal} For $s > 0$ and $r > 0$, let $g: \R^n \to \R^+$ be $s$-concave and $r$-homogeneous. Then, $g$ is also $1/r$-concave.
\end{proposition}
We have $\phi$ is $1/r$-concave, which yields $\mu$ is $1/(n+r)=1/\alpha$-concave. Thus, if $\mu$ is $p$-concave and $\alpha$-homogeneous, $\alpha\geq n$, then $\mu$ is also $1/\alpha$-concave. In this instance, using Theorem~\ref{t:min_eq} for an $\alpha$-homogeneous measure with $F(x)=x^{1/\alpha}$ yields a result similar to Equation~\ref{eq:min_ineq}:
\begin{equation}
\mu(K, L) \geq \alpha\mu(K)^{1-\frac{1}{\alpha}} \mu(L)^{\frac{1}{\alpha}}.
\label{eq:mr}
\end{equation}
This was first obtained by Milman and Rotem in \cite{MR14}, who stated it in the form of a function that is $r$-homogeneous, $1/r$-concave, that generates a measure that is $q=\frac{1}{n+r}$-concave and $(n+r)$-homogeneous. The two formulations are equivalent via Borell's classification, with $\alpha=n+r$. We conclude by stating the following.
\begin{theorem}
\label{t:min_exist}
Let a Radon measure $\mu$ be $\alpha$-homogeneous and $p$-concave, where either $\alpha\geq n$ and $p>0,$ or $\alpha<n$ but $p=\frac{1}{\alpha}$. Suppose centered $K,L\in\conbod_0$ satisfy
$$dS^{\mu}_{K}(u)=dS^{\mu}_{L}(u).$$
Then, $K=L$.
\end{theorem}
\begin{proof}
If $\alpha<n$, then $\mu$ is $1/\alpha$-concave by hypothesis. If $\alpha\geq n$, then, from Proposition~\ref{p:gal}, $\mu$ is also $1/\alpha$-concave. In either case, we can apply Remark~\ref{r:min_exist}.
\end{proof}

Theorem~\ref{t:min_exist} was first obtained by Livshyts in the case of symmetric convex bodies in \cite{GAL19} using Proposition~\ref{p:gal} and then Equation~\ref{eq:mr}. Having discussed uniqueness, we are now ready to establish Minkowski's existence theorem.
\begin{proof}[Proof of Theorem~\ref{t:third}]
We follow a technique from \cite{HLYZ10}, which, in turn, is a variation of a technique used by Chou and Wang \cite{CW06} and Lutwak \cite{LE93}. Let $\beta\in(0,\infty)$ so that $\mu\in\Lambda^n$, and define the following functional on the set of even functions in $C(\s^{n-1})$:
\begin{equation}
    \psi_{\nu,\beta}(f)=\frac{n}{\beta}\mu\left([f]\right)^\frac{\beta}{n}-\int_{\s^{n-1}}f(u)d\nu(u).
    \label{eq:functional}
\end{equation}
Notice that, for an even, positive function $f\in C(\s^{n-1}),$ $[f]$ is a symmetric convex body, and $h_{[f]} (u) \leq f(u)$ for all $u\in\s^{n-1}$. We also see that, since $[h_{[f]}]=[f],$ one obtains
\begin{align*}
    \psi_{\nu,\beta}(h_{[f]})&= \frac{n}{\beta}\mu\left([f]\right)^\frac{\beta}{n}-\int_{\s^{n-1}}h_{[f]}(u)d\nu(u)
    \\
    &\geq \frac{n}{\beta}\mu\left([f]\right)^\frac{\beta}{n}-\int_{\s^{n-1}}f(u)d\nu(u)=\psi_{\nu,\beta}(f).
\end{align*}
Thus, a maximiser of $\psi$ is the support function of a symmetric convex body. We first show that the convex body associated with a maximiser satisfies the claim, and then show the existence of the maximiser. Among the $\beta\in(0,\infty)$ so that $\mu\in\Lambda^n$, fix one of them and let $h_{K_\beta}$ be the associated maximiser of $\psi_{\nu,\beta}$; thus $K_\beta$ is symmetric. For even $f\in C(\s^{n-1}),$ pick $\delta$ small enough so that, for $|t|<\delta$, the function
$h_t(u)=h_{K_\beta}(u)+tf(u)$ is positive for all $u\in\s^{n-1}$. Viewing $\psi_{\nu,\beta}(h_t)$ as a function of one variable in $t$, we must have that $\psi_{\nu,\beta}(h_t)$ has an extremum at $t=0$, and so
$$\diff{\psi_{\nu,\beta}(h_t)}{t}\bigg|_{t=0}=0.$$
On the other hand, we see that
\begin{align*}0=\diff{\psi_{\nu,\beta}(h_t)}{t}\bigg|_{t=0}&=\mu(K_\beta)^{\frac{\beta}{n}-1}\diff{}{t}\mu([h_t])\bigg|_{t=0}-\diff{}{t}\int_{\s^{n-1}}h_{t}(u)d\nu(u)\bigg|_{t=0}
\\
&=\mu(K_\beta)^{\frac{\beta}{n}-1}\int_{\s^{n-1}}f(u)dS^{\mu}_{K_{\beta}}(u)-\int_{\s^{n-1}}f(u)d\nu(u),\end{align*}
where the second equality follows from Equation~\ref{eq:arb_mixed}.
Re-arranging, we then see
$$\int_{\s^{n-1}}f(u)d\nu(u)=\int_{\s^{n-1}}f(u)\left[\mu(K_\beta)^{\frac{\beta}{n}-1}dS^{\mu}_{K_{\beta}}(u)\right].$$
Since continuous functions are dense in $L^1(\s^{n-1})$, the above is true for all even $f\in L^1(\s^{n-1})$. We then have via the Riesz Representation theorem that
$$d\nu(u)=\mu(K_\beta)^{\frac{\beta}{n}-1}dS^{\mu}_{K_{\beta}}(u),$$
as required. We now work towards the existence of a maximiser. We will first consider dilates of the unit ball. Begin by defining the function
\begin{equation}
    q_{\beta}(\nu,r):=\psi_{\nu,\beta}\left(h_{rB_2^n}\right)=\frac{n}{\beta}\mu\left(rB_2^n\right)^\frac{\beta}{n}-r\nu(\s^{n-1}),
    \label{eq:qte}
\end{equation}  

From the condition that $\lim_{r\to\infty}\frac{\mu(rB_2^n)^\frac{\beta}{n}}{r} =0$, one has that $$\lim_{r\to\infty}q_{\beta}(\nu,r)=-\infty.$$ Since $\mu$ is a Borel measure with density, there exists some $r_0^\star$ such that $q_{\beta}(\nu,r)<0$ for $r>r_0^\star$. From the condition that $\lim_{r\to 0}\frac{\mu(rB_2^n)^\frac{\beta}{n}}{r} = \infty$, we have $r_0^\star$ is strictly greater than zero, i.e. there exists some $r_0$ such that $q_{\beta}(\nu,r)>0$ for $r\in(0,r_0).$
Since $q_{\beta}(\nu,r)$ may oscillate in sign on the interval $[r_0,r_0^\star]$, we set $R_0=\sup\{r\in\R^{+}:q_{\beta}(\nu,r)>0\}$, which is finite by the existence of $r_0^\star.$ So we have shown there exist convex bodies where $\psi_{\nu,\beta}$ is positive. We now show that it suffices to look for the maximiser contained in a (fixed) ball of some radius. Observe that the map, for $\theta\in\s^{n-1}$,
$$\theta\to\int_{\s^{n-1}}|\langle\theta,u\rangle|d\nu(u)$$
is strictly positive, since $\nu$ is not concentrated on any hemisphere. Since $\nu$ is finite, we can find a constant $C_\nu$ such that
$$\frac{1}{\nu(\s^{n-1})}\int_{\s^{n-1}}|\langle\theta,u\rangle|d\nu(u) \geq C_\nu >0.$$
Furthermore, one has that $C_\nu\leq1$ since $|\langle \theta, u\rangle|\leq 1$. 

Fix a symmetric $K\in\conbod_0$. Since $K$ is a star body, $\rho_K(\theta)$ is continuous, and so there exists some $\theta_K\in\s^{n-1}$ such that $\rho_K(\theta_K)$ is maximal. One has that the line segment $[-\rho_K(\theta_K)\theta_K,\rho_K(\theta_K)\theta_K]$ is completely contained in $K$, and yet $K$ is contained in the ball of radius $\rho_K(\theta_K).$ Furthermore, from convexity, one has $h_K(u)\geq \langle \rho_K(\theta_K)\theta_K, u\rangle$ for all $u\in\s^{n-1}$, and so from the symmetry of $K$ one has $h_K(u)\geq \rho_K(\theta_K)| \langle \theta_K, u\rangle|$. 

We can now directly compute:
\begin{align*}
\psi_{\nu,\beta}(h_K)&\leq \frac{n}{\beta}\mu(K)^\frac{\beta}{n}-\rho_K(\theta_K)\int_{\s^{n-1}}|\langle \theta_K, u\rangle|d\nu(u)
\\
&< \frac{n}{\beta}\mu(K)^\frac{\beta}{n}-\rho_K(\theta_K)C_\nu\nu(\s^{n-1})
\\
&\leq \frac{n}{\beta}\mu(\rho_K(\theta_K)B_2^n)^\frac{\beta}{n}-\rho_K(\theta_K)C_\nu\nu(\s^{n-1})
\\
&=q_{\beta}(C_\nu\nu,\rho_K(\theta_K)).
\end{align*}
Repeating the above analysis on $q_{\beta}(C_\nu\nu,\rho_K(\theta_K))$, there exists $r_1^\star,r_1>0$ such that $q_{\beta}(C_\nu\nu,r) <0$ for $r>r_1^\star$ and $q_{\beta}(C_\nu\nu,r) >0$ for $r<r_1$. We then set $R=\sup\{r\in\R^{+}:q_{\beta}(C_\nu\nu,r)>0\}$, which is finite by the existence of $r_1^\star$; we remark that, since $q_{\beta}(C_\nu\nu,r)$ dominates $q_{\beta}(\nu,r)$, one has $0<r_0\leq R_0\leq R <\infty$. It follows that $R$ is the largest possible radius where $$\psi_{\nu,\beta}\left(h_{K}\right)>0\rightarrow K\subset RB_2^n.$$ 
Thus, we can restrict our search for the maximiser to the set
$$\mathcal{F}=\{K\in\conbod_0: K=-K, K\subset R B_2^n\}.$$

We have that $\sup\{\psi_{\nu,\beta}(h_K):K\in\mathcal{F}\}>0$ by the construction of $R$. Let $\{K_i\}\subset \conbod_0$ be a sequence of convex bodies so that $\lim_{i\to\infty}\psi_{\nu,\beta}(h_{K_i})=\sup\{\psi_{\nu,\beta}(h_K):K\in\mathcal{F}\}.$ Via Blaschke selection, there exists some $K_\beta\in\mathcal{F}$ and a sub-sequence $\{K_{i_j}\}\subset\{K_i\}$ such that $\lim_{j\to\infty}\psi_{\nu,\beta}(h_{K_{i_j}})=\psi_{\nu,\beta}(h_{K_\beta})$, and so $\lim_{j\to\infty}K_{i_j}=K_\beta$ with respect to the Hausdorff metric. By construction, $\psi_{\nu,\beta}(f)\leq \psi_{\nu,\beta}(h_{K_\beta})$ for all $f\in C(\s^{n-1})$, and hence by our previous analysis, $K_\beta$ is symmetric and a maximiser. Finally, we see that, from the definition of $\mathcal{F}$,
\begin{align*}
    \frac{n}{\beta}\mu(K_\beta)^\frac{\beta}{n}=\lim_{j\to\infty}\frac{n}{\beta}\mu(K_{i_j})^\frac{\beta}{n}\geq \lim_{j\to\infty}\psi_{\nu,\beta}(h_{K_{i_j}}) > 0,
\end{align*}
and so $K_\beta$ has non-empty interior, and is thus a convex body in the proper sense.
\end{proof}
\begin{remark}
The following weaker version of Theorem~\ref{t:third} holds:
Let $\mu$ be a Borel measure with even and continuous density. Next, let $\nu$ be a finite, even Borel measure on $\s^{n-1}$ that is not concentrated on any hemisphere, such that $\mu$, $\nu$ and some $\beta\in(0,\infty)$ satisfy the following relations:
$$\lim_{r\to\infty}\frac{\mu(rB_2^n)^\frac{\beta}{n}}{r} < \nu(\s^{n-1})$$
and
there exists some $\delta>0$ such that, for every $r\in(0,\delta)$ one has
$$\frac{\mu(rB_2^n)^\frac{\beta}{n}}{r} \geq \nu\left(\s^{n-1}\right).$$
Then, there exists a symmetric convex body $K$ such that
$$d\nu(u)=\mu(K)^{\frac{\beta}{n}-1}dS^{\mu}_{K}(u).$$
\end{remark}

\noindent We now show an example on how one cannot expect to remove the constant $\mu(K)^{\frac{\beta}{n}-1}$ in general.
\begin{example}
Let $t>0$, and let $d\nu(u)=tdu.$ Then for $t>0$ large enough, there does not exist a $K\in\conbod$ satisfying $$d\nu(u)=dS^{\gamma_n}_{K}(u).$$ Indeed, $S^{\gamma_n}_{K}(\s^{n-1})$ is bounded above by $4n^{\frac{1}{4}}$ \cite{Ball93, Naz03}, but $\nu(\s^{n-1})$ can be made arbitrarily large.
\end{example}

We next provide an example where the constant $\mu(K)^{\frac{\beta}{n}-1}$ can be removed; we shall show the case of a Borel measure $\mu$ with density such that $\mu$ is $\alpha$-homogeneous, $\alpha\in(0,1)\cup(1,\infty)$.
\begin{lemma}
\label{l:almost_hom}
Let a Borel measure $\mu$ with density be even, $\alpha$-homogeneous, $\alpha > 0,$ $\alpha\neq 1$. Suppose $\nu$ is a finite, even measure on $\s^{n-1}$ not concentrated on any hemisphere. Then, there exists a symmetric $K\in\conbod_0$ such that
$$d\nu=dS^{\mu}_{K}.$$
\end{lemma}
\begin{proof}
From the homogeneity we obtain that
$$\frac{\mu(rB_2^n)^\frac{\beta}{n}}{r}=\mu(B_2^n)r^{\frac{\beta\alpha}{n}-1}.$$
Thus, we must have $\beta < \frac{n}{\alpha}$ for $\mu$ to be in $\Lambda^n$. Set, for example, $\beta=\frac{n}{2\alpha}.$ Hence, by Theorem~\ref{t:third}, there exists a symmetric convex body $\tilde{K}$ such that
$$d\nu(u)=\mu(\tilde{K})^{\frac{\beta}{n}-1}dS^{\mu}_{\tilde{K}}(u).$$

We have that $dS^{\mu}_{L}(u)$ is $(\alpha-1)$-homogeneous (in the argument $L$) for $L\in\conbod$. Indeed, it suffices to show for $L$ of class $C^2_+$. Then, one has that $t^{n-1}dS_L(u)=dS_{tL}(u)$ for $t>0$.  Thus, we obtain, using that $n^{-1}_{tL}(u)=tn^{-1}_{L}(u)$ and the $(\alpha-n)$-homogeneity of $\phi$,
$$dS^{\mu}_{tL}(u)=\phi\left(n^{-1}_{tL}(u)\right)dS_{tL}(u)=t^{n-1}t^{\alpha-n}\phi\left(n^{-1}_{L}(u)\right)dS_{L}(u)=t^{\alpha-1}dS^{\mu}_{L}(u).$$
Next, let $A$ be such that
$$A^{\alpha-1}=\mu(\tilde{K})^{\frac{\beta}{n}-1},$$
that is $A=\mu(\tilde{K})^{\frac{\beta-n}{n(\alpha-1)}}$. Set $K=A\tilde{K}$. We directly compute:
$$d\nu(u)=A^{\alpha-1}dS^{\mu}_{\tilde{K}}(u)=dS^{\mu}_{A\tilde{K}}(u)=dS^{\mu}_{K}(u),$$
as claimed. 
\end{proof}

\begin{remark}
For a measure that is $1$-homogeneous, the surface area measure $dS^{\mu}_{L}(u)$ is $0$-homogeneous; therefore, the above schema fails, and this case remains open.
\end{remark}

Combining existence and uniqueness, we can extend the result of Livshyts \cite{GAL19}:

\begin{theorem}
\label{t:min_exist_uni}
Let a Radon measure $\mu$ be even, $p$-concave, and $\alpha$-homogeneous so that $\alpha \neq 1$ and either $\alpha \geq n$, $p>0$ or $\alpha<n$ but $p=\frac{1}{\alpha}$.  Suppose $\nu$ is a finite, even measure on $\s^{n-1}$ not concentrated on any hemisphere. Then, there exists a unique symmetric $K\in\conbod_0$ such that
$$d\nu=dS^{\mu}_{K}.$$
\end{theorem}
\begin{proof}
From Lemma~\ref{l:almost_hom}, we know there exists a symmetric $K\in\conbod_0$ such that $d\nu=dS^{\mu}_{K}$. Suppose there exists another symmetric $L \in\conbod_0$ such that $d\nu=dS^{\mu}_{L}$. Then $dS^{\mu}_{L}=dS^{\mu}_{K}$, and so $K=L$ from Lemma~\ref{l:min_exist}.
\end{proof}

We conclude this section by remarking that it was brought to our attention by Zvavitch that some of the above endeavours have been previously done specifically for the Gaussian measure $\gamma_n$ in \cite{HXZ21}. In particular, Theorem~\ref{t:ehr} is presented following the same calculations as Lemma~\ref{l:uni}. Lemma~\ref{l:second} and Theorem~\ref{t:third} also appear specifically for the Gaussian measure, but both proofs presented in \cite{HXZ21} use the measure's specific properties. Therefore, despite both referencing the same technique from \cite{HLYZ10}, the proofs presented in \cite{HXZ21} seemingly do not generalize directly to what is presented here.

\section{Classification Results for Weighted Projection Bodies}
\label{s:class}
We begin this section with a discussion of the cosine transform.
\begin{definition}
\label{def:cos}
Let $f\in L^1(\s^{n-1})$. Then, the \textit{cosine transform} of $f$ is given by
$$\ct{f}{\theta}=\int_{\s^{n-1}}|\langle\theta,u\rangle|f(u)du.$$
For a signed Borel measure $\nu$ on $\s^{n-1}$, the cosine transform of $\nu$ is given by
$$\ct{d\nu}{\theta}=\int_{\s^{n-1}}|\langle\theta,u\rangle|d\nu.$$
\end{definition}

The cosine transform also has a well-known uniqueness property, which we list in the following lemma. We denote $f^+$ to be the even part of $f$, that is $$f^+(x)=\frac{f(x)+f(-x)}{2}.$$
\begin{lemma}[Funk-Hecke Theorem for the Cosine Transform, \cite{Gr}]
\label{l:funk}
Suppose $F_1,F_2$ are two bounded, integrable functions or two signed measures on $\s^{n-1}$. One has that $$\ct{F_1}{\theta}=\ct{F_2}{\theta}$$
for every $\theta\in\s^{n-1}$ if, and only if, $F_1^+(u)=F_2^+(u)$ for almost all $u\in\s^{n-1}.$
\end{lemma}

The cosine transform is intimately related with the Fourier transform. The following can be found in \cite{AK05,KRZ04_1,KRZ04_2, KY08}. The set $\mathcal{S}$ denotes the Schwarz class of functions, that is the class of functions that are infinitely differentiable, rapidly decreasing and go from $\R^n\to\mathbb{C}.$ A function $f\in L^1_{loc}(\R^n)$ acts on $\mathcal{S}$ via integration: for $\psi\in\mathcal{S}$ one has $\langle f,\psi\rangle_{\mathcal{S}}=\int_{\R^n}f(y)\psi(y)dy.$ The Fourier transform of $f$, denoted $\widehat{f}$, is the function defined via
$$\langle \widehat{f},\psi\rangle_{\mathcal{S}}=\langle f,\widehat{\psi}\rangle_{\mathcal{S}},\quad \text{and satisfies }\langle \widehat{f},\widehat{\psi}\rangle_{\mathcal{S}}=(2\pi)^n\langle f,\psi\rangle_{\mathcal{S}}$$ for every $\psi\in\mathcal{S},$ where $\widehat{\psi}$ is the standard Fourier transform of $\psi$. In general, $\widehat{f}$ will be a distribution that acts on Schwarz functions via integration. For $\nu$ a Borel measure on $\s^{n-1}$, one has the $(-n-1)$-homogeneous extension of $\nu$, the distribution denoted $\nu_e$, is defined via its action
$$\langle d\nu_e,\psi\rangle_{\mathcal{S}}=\frac{1}{2}\int_{\s^{n-1}}\langle r^{-2},\psi(r\xi)\rangle_{\mathcal{S}}d\nu(\xi),$$
where, if $0\notin \text{supp}(\psi)$,
$$\langle r^{-2},\psi(r\xi)\rangle=\int_{\R}r^{-2}\psi(r\xi)dr.$$
If the density of $\nu$ is $g\in L^1_{loc}(\s^{n-1})$, then the $(-n-1)$-homogeneous extension of $g$ is, by slightly abusing notation, $g(x)=|x|^{-n-1}g(x/|x|)$. In this instance, the Fourier transform of $d\nu_e$ is the Fourier transform of $g.$ Throughout the remainder of this paper, the Fourier transform of a measure or a function on the sphere will always be understood as the Fourier transform of the $(-n-1)-$homogeneous extension.
We now state a well known lemma concerning the Fourier transform of a measure on the sphere from \cite{KRZ04_2}.
\begin{lemma}[Fourier Transform of Measures]
Consider an even Borel measure $\nu$ on $\s^{n-1}$ with density $g$. Then, the Fourier transform of $\nu_e$ is given by
\begin{equation*}
\widehat{\nu_e}(\theta)=-\frac{\pi}{2}\ct{d\nu}{\theta}=-\frac{\pi}{2}\int_{\s^{n-1}}|\langle\theta,u\rangle|d\nu(u).
\end{equation*}
This implies for an even, non-negative function $g$ on $\s^{n-1}$ that
\begin{equation*}
\widehat{g}(\theta)=-\frac{\pi}{2}\ct{g}{\theta}=-\frac{\pi}{2}\int_{\s^{n-1}}|\langle\theta,u\rangle|g(u)du.
\end{equation*}
\label{lem:for_cos}
\end{lemma}

We conclude our discussion by stating a Parseval Formula on the sphere, proven by Koldobsky, Ryabogin, and Zvavitch \cite{KRZ04_1}.
\begin{lemma}
\label{l:par}
Let $K, L\in\conbod_0$ be symmetric and $C^2_+$, so that $dS_K=f_K du$. Then, one has
$$\int_{S^{n-1}} \widehat{h_{L}}(\theta) \widehat{f_{K}}(\theta) d \theta=(2 \pi)^{n} \int_{S^{n-1}} h_{L}(\theta) f_{K}(\theta) d \theta.$$
\end{lemma}

We next state the (well known, say in \cite{gardner_book}) classification results for $\Pi K$, which we recall is the symmetric convex body whose support function is given by
$$h_{\Pi K}(\theta)=\frac{1}{2}\int_{\s^{n-1}}|\langle \theta,u\rangle|dS_K(u)=\frac{1}{2}\ct{dS_K}\theta$$ for $K\in\conbod$,
in order to then classify the weighted projection bodies $\Pi_{\mu} K$. Fix $K\in\conbod$. It is not true in general that if $L\in\conbod$, then $\Pi K = \Pi L$ implies $K=L$. Hence, one defines the \textit{projective class} of $K$, denoted $\mathcal{P}K$ to be the set of convex bodies such that $L\in\mathcal{P}K\leftrightarrow \Pi L= \Pi K.$ The \textit{Blaschke body} of $K$, denoted $\nabla K$, is defined \cite{gardner_book, LE98, CMP67} as the symmetric convex body whose surface area measure is given by
$$dS_{\nabla K}(\theta)=\frac{1}{2}(dS_{K}(\theta) + dS_{-K}(\theta)).$$
One has that $\nabla K$ exists, via Minkowski's Existence Theorem, and that $\vol_n(\nabla K)\geq \vol_n(K)$, with equality if, and only if, $K$ is symmetric. Indeed, we have that
\begin{align*}
\vol_n(\nabla K)&=\frac{1}{n}\int_{\s^{n-1}}h_{\nabla K}(u)dS_{\nabla K}(u)
\\
&=\frac{1}{2n}\left(\int_{\s^{n-1}}h_{\nabla K}(u)dS_{K}(u)+\int_{\s^{n-1}}h_{\nabla K}(u)dS_{-K}(u)\right)
\\
&=\frac{1}{n}\int_{\s^{n-1}}h_{\nabla K}(u)dS_{K}(u)=V(K,\nabla K).
\end{align*}
Thus, applying Minkowski's inequality yields
$\vol_n(\nabla K)^n=V(K,\nabla K)^n\geq \vol_n(K)^{n-1}\vol_n(\nabla K)$. Re-arranging yields the result.
We see from direct substitution that $\Pi(\nabla K)=\Pi K.$ We collect these classical results in the following proposition.
\begin{proposition}
\label{p:class}
Let $K\in\conbod$ with projective class $\mathcal{P}K$. Then, there exists a unique convex body in $\mathcal{P}K$, $\nabla K$, that is symmetric. Furthermore, $\nabla K$ is the largest with respect to volume.
\end{proposition}


Projection bodies also admit a Fourier classification, proven for the case of symmetric convex bodies in \cite{KRZ04_1,KRZ04_2}. We recall this here and expand their results to the non-symmetric case. In Lemma~\ref{lem:for_cos}, consider the case of $d\nu(u)=dS_K(u)$ where $K$ is a symmetric convex body of the class $C^2_+$. If one takes $g(u)=f_K(u)$, which is the density of the surface area measure of $K$, or curvature function, one obtains
$\widehat{f_K}=-\pi h_{\Pi K}.$ Dropping the $C^2_+$ assumption yields $\widehat{(S_K)_e}=-\pi h_{\Pi K}$. Notice that
$$\int_{\s^{n-1}}|\langle\theta,u\rangle|dS_K(\theta)=\int_{\s^{n-1}}|\langle\theta,u\rangle|dS_{\nabla K}(\theta).$$
Therefore, if $K$ is not $C^2_+$ and not necessarily symmetric, then one obtains in general that
$$\widehat{(S_{\nabla K})_e}(\theta)=-\pi h_{\Pi K}(\theta).$$
Next, observe that, for every even test function $\psi\in\mathcal{S}$:
\begin{align*}\langle \widehat{h_{\Pi K}},\psi \rangle_{\mathcal{S}}&=\langle h_{\Pi K},\widehat{\psi} \rangle_{\mathcal{S}}
=-\frac{1}{\pi}\langle \widehat{(S_{\nabla K})_e},\widehat{\psi} \rangle_\mathcal{S}
=\langle -2(2\pi)^{n-1}(S_{\nabla K})_e ,\psi\rangle_{\mathcal{S}}.\end{align*}
Since $\widehat{h_{\Pi K}}$ and $-2(2\pi)^{n-1}(S_{\nabla K})_e$ are both even, homogeneous and agree on the space of even test functions, one obtains
$$\widehat{h_{\Pi K}}(\theta)=-2(2\pi)^{n-1}(S_{\nabla K})_e(\theta),$$
in particular, if $K$ is $C^2_+$ and symmetric,
$\widehat{h_{\Pi K}}(\theta)=-2(2\pi)^{n-1}f_{K}(\theta).$

Furthermore, one has that a symmetric convex body $L$ is a projection body if, and only if, $\widehat{h_L}=-2(2\pi)^{n-1}\nu_e$ for some even Borel measure on $\s^{n-1}.$ If $L=\Pi K$ for some $K\in\conbod$, then $\widehat{h_{\Pi K}}(\theta)=-2(2\pi)^{n-1}(S_{\nabla K})_e(\theta)$ for all $\theta\in\s^{n-1}$; conversely, if one has $\widehat{h_L}=-2(2\pi)^{n-1}\nu_e$, then this yields $h_L(\theta)=\frac{1}{2}\int_{\s^{n-1}}|\langle u,\theta \rangle|d\nu(u)$. From the fact that $L$ is a full dimensional convex body, $\nu$ is not concentrated on any subsphere and so $\nu=S_K$ for some convex body $K$ via Minkowski's existence theorem, which means $L=\Pi K$. We will use this fact and the cosine transform to obtain the following classification lemma, which can easily be shown using the above discussion e.g. in \cite{KRZ04_2}.
\begin{lemma}[Proposition 1 in \cite{KRZ04_2}]
\label{l:proj_meas}
A convex body $L\in\conbod$ is a projection body if, and only if, there exists a finite Borel measure $\nu$ on $\s^{n-1}$ such that, for every even, non-negative $g\in C(\s^{n-1})$, one has
$$\int_{\s^{n-1}} h_L(u) g(u) du = \int_{\s^{n-1}} \ct{g}{u} d\nu(u).$$
\end{lemma}

We now begin the classification of $\Pi_\mu K$, that is the symmetric convex body whose support function is given by, for Borel $\mu$ with density $\phi$ containing $\partial K$ in its Lebesgue set for a convex body $K$ and $\theta\in\s^{n-1}$:
$$h_{\Pi_{\mu} K}(\theta)=\frac{1}{2}\int_{\s^{n-1}}|\langle u,\theta\rangle|dS^{\mu}_{K}(u).$$ From the definition of the cosine transform, Definition~\ref{def:cos}, we see that we can write
\begin{equation}
h_{\Pi_{\mu} K}(\theta)=\frac{1}{2}\ct{S^\mu_K}{\theta}=\frac{1}{2}\ct{\left(\phi\circ n_K^{-1}\right)f_K}{\theta},
\label{eq:supp_cos}
\end{equation}
where the last equality holds if $K$ is $C^2_+$, i.e. if $dS_K(u)=f_K(u)du$ and $n_K:\partial K\to \s^{n-1}$ is a bijection.
\begin{definition}[Weighted Projective Classes]
Suppose $\mu$ is a Borel measure with continuous density. Then, for $K\in\conbod$, the projective class of $K$ with respect to $\mu$ is the subset of $\conbod$ denoted by $\mathcal{P}_\mu K.$ That is, $L\in\mathcal{P}_\mu K \leftrightarrow \Pi_\mu L = \Pi_\mu K.$ 
\end{definition}
From here, we define weighted Blaschke bodies.
\begin{definition}
For a Borel measure on $\R^n$ with locally integrable density containing $\partial K$ in its Lebesgue set for a given $K\in\conbod$, the \textit{$\mu$-Blaschke body} $\nabla_\mu K$ of $K$ is the convex body whose $\mu$-surface area measure is given by
$$dS^{\mu}_{\nabla_\mu K}(u)=\frac{1}{2}\left(dS^{\mu}_{K}(u) + dS^{\mu}_{K}(-u)\right).$$
\end{definition}
For general Borel measures $\mu$, $\nabla_\mu K$ may not be well-defined, that is, it is not necessarily true that $dS^{\mu}_{\nabla_\mu K}(u)$ is the $\mu$-surface area measure of a convex body. If $\nabla_\mu K$ is well-defined, then via direct substitution, we get $\Pi_\mu K = \Pi_\mu \nabla_\mu K$, and hence $\nabla_\mu K\in \mathcal{P}_\mu K$. Also, $\nabla_\mu K$ is symmetric, and so $h_{\nabla_\mu K}(-u)=h_{\nabla_\mu K}(u)$.
\begin{proposition}
\label{p:even_class}
Consider a Radon measure $\mu$ such that $\mu$ is $p$-concave and $\alpha$- homogeneous. Let a convex body $K$ have projective class $\mathcal{P}_\mu K$. Then, there exists a convex body in $\mathcal{P}_\mu K$, $\nabla_\mu K$, that is symmetric. Furthermore, $\nabla_\mu K$ is the largest with respect to the $\mu$-measure in the following sense:
    If $\alpha \geq n$ and $p>0$ or if $\alpha < n$ but  $p\alpha =1$, then  $\nabla_\mu K$ is unique and $\mu(\nabla_\mu K)\geq \mu(K)$ with equality if, and only if, $K$ is symmetric.
\end{proposition}
\begin{proof} Fix some convex body $K$. From the extension of Minkowski's existence theorem, Theorem~\ref{t:min_exist}, $\nabla_\mu K$ is well-defined for $\mu$ $\alpha$-homogeneous, $p$-concave, $\alpha\in(0,1)\cup(1,\infty)$, but, if $\alpha<n$ and $\alpha p <1$, $\nabla_\mu K$ is not necessarily unique. 
From Equation~\ref{eq:arb_mixed}, we see that,
\begin{align*}\mu(K,\nabla_\mu K)&=\int_{\s^{n-1}}h_{\nabla_\mu K}(u)dS^{\mu}_{K}(u)
\\
&=\frac{1}{2}\left(\int_{\s^{n-1}}h_{\nabla_\mu K}(u)dS^{\mu}_{K}(u)+\int_{\s^{n-1}}h_{\nabla_\mu K}(u)dS^{\mu}_{K}(u)\right)
\\
&=\frac{1}{2}\left(\int_{\s^{n-1}}h_{\nabla_\mu K}(u)dS^{\mu}_{K}(u)+\int_{\s^{n-1}}h_{\nabla_\mu K}(-u)dS^{\mu}_{K}(-u)\right)
\\
&=\mu(\nabla_\mu K,\nabla_\mu K).
\end{align*}

\noindent Using Minkowski's inequality for general measures, that is Lemma~\ref{t:min_eq}, then yields
$$\mu(\nabla_\mu K,\nabla_\mu K)=\mu(K,\nabla_\mu K)\geq \mu(K,K) + \frac{\mu(\nabla_\mu K)^p -\mu(K)^p}{p\mu(K)^{p-1}}.$$
Multiplying both sides by $t^{\alpha-1}$ and integrating over $t\in[0,1]$ then yields
$$\mu(\nabla_\mu K)\geq \mu(K) + \frac{\mu(\nabla_\mu K)^p -\mu(K)^p}{\alpha p\mu(K)^{p-1}}.$$
Re-arranging then yields that
$$(1-\alpha p)\mu(K)^{p} + \alpha p\mu(K)^{p-1}\mu(\nabla_\mu K)\geq \mu(\nabla_\mu K)^p. $$
If $\alpha p =1$, then one has $\mu(\nabla_\mu K)\geq \mu(K)$ with equality if, and only if, $K$ is symmetric, from Proposition~\ref{p:mr} and Corollary~\ref{cor:minkowski}.
In the case where $\alpha\geq n$, we can use Equation~\ref{eq:mr} to obtain the same result.
\end{proof}

For $\mu\in\Lambda^n$, one could define the Blaschke body in the following way.
\begin{definition}
Let $\mu\in\Lambda^n$ with associated $\beta>0$ and fix $K\in\conbod.$ The \textit{$\beta$-$\mu$-Blaschke body} $\nabla_{\mu,\beta} K$ of $K$ is the symmetric convex body whose $\mu$-surface area measure is given by
\begin{equation}
\label{eq:blaschke}
\begin{split}
\mu\left({\nabla}_{\mu,\beta} K\right)^{\frac{\beta}{n}-1}dS^{\mu}_{\nabla_{\mu,\beta} K}(u)
=\frac{\mu(K)^{\frac{\beta}{n}-1}}{2}\left(dS^{\mu}_{K}(u) + dS^{\mu}_{K}(-u)\right)
\end{split};
\end{equation}
from Theorem~\ref{t:third}, ${\nabla}_{\mu,\beta} K$ exists.
\end{definition}

\noindent  We see that we lose uniqueness of the Blaschke body in such a case (as for many measures, if they satisfy the hypotheses of Theorem~\ref{t:third} for some $\beta>0$, then they satisfy the hypotheses for $\beta^\prime$ near $\beta$ as well). One would like to show that $\nabla_{\mu,\beta}K$ is the "largest" convex body in some sense of $\mathcal{P}_{\mu}K$. Unlike in the proof of Proposition~\ref{p:even_class}, using Lemma~\ref{t:min_eq} in the case of $F$-concave $\mu$ for some differentiable $F$ does not yield the result. However, we can remark that, if there exists a $\beta^\star\in(0,\infty)$ so that the $\beta^\star$-$\mu$-Blaschke body $\nabla_{\mu,\beta^\star} K$ defined in Equation~\ref{eq:blaschke} satisfies
$$\mu(\nabla_{\mu,\beta^\star} K)=\mu(K),$$
then, Lemma~\ref{t:min_eq} yields
$$\mu(\nabla_{\mu,\beta^\star} K,\nabla_{\mu,\beta^\star} K) \geq \mu(K,K).$$

Just like in the case of projection bodies, the Fourier transform of the support function of weighted projection bodies is related to the surface area measure. Suppose $K$ is a symmetric convex body of class $C^2_+,$ that $\mu$ has even, continuous density $\phi$, and denote $\phi_K(u):=f_K(u)\phi(n_K^{-1}(u))$. Then, setting $g(u)=\phi_K(u)$ in Lemma~\ref{lem:for_cos} yields
$\widehat{\phi_K}(\theta)=-\pi h_{\Pi_{\mu}K}(\theta).$ Arguing similarly to the volume case, this then implies that
$
\widehat{h_{\Pi_\mu K}}(\theta)=-2(2\pi)^{n-1}\phi_K(\theta).
$
Dropping the $C^2_+$ and symmetry assumptions, for a Borel measure $\mu$ on $\R^n$ with density $\phi$ containing $\partial K$ in its Lebesgue set for some $K\in\conbod$, one obtains
$$\widehat{h_{\Pi_\mu K}}(\theta)=-2(2\pi)^{n-1}(S^\mu_{\nabla_\mu K})_e(\theta) \quad \text{and} \quad \widehat{(S^\mu_{\nabla_\mu K})_e}(\theta)=-\pi h_{\Pi_\mu K}(\theta).$$

We conclude this section by comparing the projection bodies $\Pi_\mu K$ with another generalization of projection bodies to the weighted setting. In \cite{GAL19}, an alternative generalization of a projection body for a symmetric convex body $K$ and a Borel measure $\mu$ on $\R^n$ with even, continuous density $\phi$ was introduced as the convex body whose support function is given by 
$$P_{\mu,K}(\theta):=n\int_0^1 h_{\Pi_\mu tK}(\theta) dt=\frac{n}{2}\int_0^1\int_{\s^{n-1}}|\langle u,\theta \rangle|dS^{\mu}_{tK}(u)dt,$$
and, furthermore, the following Shephard problem was then solved for the case when $\mu$ was $\alpha$-homogeneous, $\alpha \geq n$, and $p$-concave, with $K,L\in\conbod_0$ symmetric such that $L$ is a projection body:
$$P_{\mu,K}(\theta) \leq P_{\mu,L}(\theta) \; \forall \; \theta\in\s^{n-1}\longrightarrow \mu(K)\leq \mu(L).$$
The schema used to solve the above Shephard problem in \cite{GAL19} was to first suppose that
$$h_{\Pi_\mu tK}(\theta)\leq h_{\Pi_\mu tL}(\theta)$$
for all $t\in[0,1]$ and $\theta\in\s^{n-1}.$ Then, it was shown this implies that $\mu(tK,L)\leq \mu(tL,L)$; by integrating over $(0,1)$ in the appropriate places, one obtains
$$P_{\mu,K}(\theta) \leq P_{\mu,L}(\theta) \; \forall \; \theta\in\s^{n-1} \longrightarrow \int_0^1\mu(tK,L)dt\leq \mu(L).$$
From the fact that $\mu^{1-q}(K)\mu(L)^q\leq \int_0^1\mu(tK,L)dt$ for an appropriately chosen $q<1$ via Proposition~\ref{p:gal}, the solution was obtained. To relate $\Pi_\mu K$ and the convex body given by $P_{\mu,K}(\theta)$, we see that we must alter $h_{\Pi_\mu K}(\theta)$ into $h_{\Pi_\mu tK}(\theta)$ and then integrate over $(0,1)$. However, such a transformation alters $\mu.$ Indeed, suppose $K$ is of class $C^2_+$ and write:
\begin{align*}
t^{n-1}h_{\Pi_\mu K}(\theta)&=\frac{t^{n-1}}{2}\int_{\s^{n-1}}|\langle u,\theta \rangle|\phi\left(n_{K}^{-1}(u)\right)dS_{K}(u)dt
\\
&=\frac{1}{2}\int_{\s^{n-1}}|\langle u,\theta \rangle|\phi\left(t^{-1}n_{tK}^{-1}(u)\right)dS_{tK}(u)dt=h_{\Pi_{\mu^\prime}tK}(\theta)
\end{align*}
where $\mu^\prime$ is the measure with density $\phi(t^{-1}x).$ Thus, in general it is not possible to do such a transformation. However, if $\mu$ is also $\alpha$-homogeneous, like in \cite{GAL19}, then $\phi\left(t^{-1}n_{tK}^{-1}(u)\right)=t^{n-\alpha}\phi\left(n_{tK}^{-1}(u)\right)$ and so
$t^{\alpha-1}h_{\Pi_\mu K}(\theta)=h_{\Pi_\mu tK}(\theta)$. Integrating both sides over $(0,1)$ then yields
$h_{\Pi_\mu K}(\theta)=\frac{\alpha}{n}P_{\mu,K}(\theta).$
Hence, this relation allows us to obtain the results of the Shephard problem for $P_{\mu,K}(\theta)$ listed above as a particular case of the isomorphic Shephard problems discussed in the next section.

\section{Shephard Problems for Weighted Projection Bodies}
\label{s:shep_meas}
In this section, we discuss Shephard problems for $\Pi_\mu K$, that is various solutions to Question~\ref{q:shep_meas}. We remind the reader that by Shephard problem, we mean given Borel measures $\mu,\nu$ with continuous densities and fixed symmetric $K,L\in\conbod_0$, we wish to find an optimal quantity $\mathcal{A}$ such that $\mu(K)\leq\mathcal{A}\nu(L)$ under the hypothesis
$$h_{\Pi_{\mu} K}(\theta)\leq h_{\Pi_{\nu} L}(\theta)$$ for all $\theta\in\s^{n-1}.$ We will not assume that $L$ is a projection body; we will give instances where the quantity $\mathcal{A}$ can be made independent of the measures $\nu$ and $\mu$ and the convex bodies $L$ and $K$. In the case where $\mu=\nu$ and $L$ is a projection body, we recover the result of \cite{GAL19}, itself an extension of the classical solution in the instance of volume, where the quantity $\mathcal{A}=1$ in this case.

Our approach will use standard Fourier techniques. It suffices to consider the case when $K$ and $L$ are of class $C^2_+$, yielding $dS_K(u)=f_K(u)du$ and similarly for $L$; the general case follows by approximation. We next recall the Banach-Mazur distance \cite{AGA}: for $M\in\conbod,$ one has
$d_{BM}(L,M)=\inf\{d>0: \exists \; T\in GL_{n} : L\subset TM \subset dL\}$. From here, we can define
$$d_{\Pi}(L)=\inf\{d_{BM}(L,M):\text{M is a projection body}\}.$$
The theorem of Fritz John (\cite{AGA, gardner_book}) states that given a symmetric $L\in\conbod_0$ there exists an ellipsoid $\mathcal{E}$ such that $\mathcal{E}\subset L \subset \sqrt{n}\mathcal{E}$. Hence, we have from John's theorem that $1\leq d_{\Pi}(L)\leq \sqrt{n}$, since $L$ is symmetric and symmetric ellipsoids are projection bodies (and so $d_{\Pi}(L)$ is well defined). One also has that the the linear image of a projection body is a projection body: for $T\in GL_{n}$ and $M\in\conbod$, one has, where $T^{-t}$ denotes the inverse transpose of $T$,
$T\Pi M = |\text{det } T|\Pi(T^{-t} M)=\Pi(|\text{det } T|^{\frac{1}{n-1}}T^{-t} M)$ \cite{gardner_book,LRZ22}.

\begin{remark}
The unit ball $B_2^n$ is a projection body, and so $$d_\Pi(L)\leq d_{BM}(L,B_2^n).$$ It was shown in \cite{DL78,SZ01} that, if $L$ is the unit ball of a subspace of $L^q, q>2$, then $d_{BM}(L,B_2^n)\leq n^{1/2-1/q}$. So throughout this section, the estimate $d_\Pi(L)\leq n^{1/2-1/q}$ can be made in this instance.
\end{remark}
\begin{lemma}
\label{l:q_0}
Let $\mu$ and $\nu$ be Borel measures with continuous densities $\phi$ and $\psi$ respectively. Suppose we have symmetric $K,L\in\conbod_0$ such that
$$h_{\Pi_\mu K}(\theta)\leq h_{\Pi_\nu L}(\theta)$$ for all $\theta\in\s^{n-1}.$ Then, one has that
$\mu(K,L)\leq d_{\Pi}(L)\nu(L,L).$
\end{lemma}
\begin{proof}
Begin by approximating $L$ by some projection body $M$ in the sense that $$M\subset L \subset d_{\Pi}(L) M.$$ By approximation, we may also assume $K$ and $L$ are of class $C^2_+$. Then, our hypotheses are equivalent to
 $$\ct{\left(\phi\circ n_K^{-1}\right)f_K}{\theta} \leq \ct{\left(\psi\circ n_L^{-1}\right)f_L}{\theta}.$$ 
  Notice that the odd parts of the functions $\phi\circ n_K^{-1}$ and $\psi\circ n_L^{-1}$ integrate to zero. That is, we may assume these functions are even.
 
 Integrating both sides of the above with respect to the Borel measure associated with $h_M$ from Lemma~\ref{l:proj_meas} yields
$$\int_{\s^{n-1}}h_M(u)\phi\left(n_{K}^{-1}(u)\right) f_{K}(u) du \leq \int_{\s^{n-1}} h_M(u) \psi\left(n_L^{-1}(u)\right)f_L (u)du.$$
Using the conditions on $M$, we have
\begin{equation}\frac{1}{d_{\Pi}(L)}\int_{\s^{n-1}}h_L(u)\phi\left(n_{K}^{-1}(u)\right) f_{K}(u) du \leq \int_{\s^{n-1}} h_L(u) \psi\left( n_L^{-1}(u)\right)f_L (u)du,
\label{eq:shep}
\end{equation}
which yields, according to the definition of mixed measures, Equation~\ref{eq:arb_mixed},
$$
  \mu(K,L) \leq d_{\Pi}(L)\nu(L,L).
$$
\end{proof}
We obtain a simplification when imposing homogeneity conditions on $\mu$ and $\nu$.
\begin{lemma}
\label{l:q_1}
Let $\mu$ and $\nu$ be Borel measures with density such that $\nu$ is $\beta$-homogeneous, $\beta>0$. If $K,L\in\conbod_0$ are symmetric such that $$h_{\Pi_{\mu} K}(\theta) \le h_{\Pi_{\nu} L}(\theta)$$ for every $\theta\in\s^{n-1},$ then,
$\mu(K,L) \leq \beta d_{\Pi}(L)\nu(L).$
\end{lemma}
\begin{proof}
From Lemma~\ref{l:q_0}, we have that
$$
  \mu(K,L) \leq d_{\Pi}(L)\nu(L,L).
$$
Since $\nu$ is $\beta$-homogeneous, it follows that the mixed measure of $\nu$ is $(\beta-1)$-homogeneous in the first variable. Thus, the inequality
$$
 t^{\beta-1}\mu(K,L) \leq  t^{\beta-1}d_{\Pi}(L)\nu(L,L)
$$
yields 
\begin{equation}
   t^{\beta-1}\mu(K,L) \leq d_{\Pi}(L)\nu(tL,L).
     \label{eq:hom_dens}
\end{equation}
Integrating both sides of Equation~\ref{eq:hom_dens} over $(0,1)$ yields the result.
\end{proof}
We see that the result in Lemma~\ref{l:q_1} is in some sense sharp. Indeed, consider first the case when $\mu=\nu=\lambda$. Then, Lemma~\ref{l:q_1} yields
$$V(K,L)=\frac{1}{n}\int_{\s^{n-1}}h_L(u) dS_K(u) \leq d_{\Pi}(L) \vol_n(L),$$
or, taking the $n$th power of both sides and using Minkowski's inequality:
$$\vol_n(K)^{n-1}\vol_n(L)\leq d^n_{\Pi}(L) \vol^n_n(L).$$
Consequently, we have
$$\vol_n(K)\leq d^{\frac{n}{n-1}}_{\Pi}(L) \vol_n(L).$$
For $n\geq 3,$ consider the estimate from John's Theorem $d_{\Pi}(L)\leq \sqrt{n}$ to obtain
\begin{equation}
\label{eq:keith}
\vol_n(K)\leq n^{\frac{n}{2n-2}} \vol_n(L)=n^{1/2}n^{\frac{1}{2n-2}}\vol_n(L)\leq 3^{1/4}n^{1/2}\vol_n(L),\end{equation}
where we used that $n^{\frac{1}{2n-2}}$ is monotonically decreasing (to $1$) as a function of $n.$ We see that Equation~\ref{eq:keith} is asymptotically sharp due to the result from Ball \cite{Ball91_0}.
We further remark that if $n=2$ or $n\geq 3$ and $L$ is is a projection body, then $d_{\Pi}(L)=1$ and so we obtain the classical solution to the Shephard problem:
$$\vol_n(K)\leq \vol_n(L).$$
We consider solutions to Question~\ref{q:shep_meas} for two different measures. We first investigate how to relate $\mu(K)$ and $\mu(L)$ under concavity assumptions on $\mu$, and assume $\mu$ and $\nu$ have some homogeneity.

\begin{proposition}
\label{p:q_4}
Suppose $\mu$ is a $p$-concave and $\alpha$- homogeneous Radon measure, and $\nu$ is a $\beta$-homogeneous Borel measure, $\alpha,\beta > 0$. Let $K$ and $L$ be symmetric convex bodies such that
$$h_{\Pi_\mu K}(\theta)\leq h_{\Pi_\nu L}(\theta)$$ for all $\theta\in\s^{n-1}.$ Then, one has that
$$\mu(K) \leq \left[d_{\Pi}(L)\beta p\frac{\nu(L)}{\mu(L)} +\frac{\mu(K)(1-\alpha p)}{\mu(L)}\right]^{\frac{1}{1-p}} \mu(L).$$
In particular, when $\alpha p=1$:
$$\mu(K) \leq \left(\frac{\beta\nu(L)}{\alpha\mu(L)}d_{\Pi}(L)\right)^{\frac{\alpha}{\alpha-1}}\mu(L).$$
\end{proposition}
\begin{proof}
From Lemma~\ref{l:q_1} we have that $\mu(K,L) \leq \beta d_{\Pi}(L)\nu(L).$ Using the extended Minkowski inequality from Lemma~\ref{t:min_eq} for $F(t)=t^p$, then yields
$$\mu(K,K)+\frac{\mu(L)^p-\mu(K)^p}{p\mu(K)^{p-1}}\leq \beta d_{\Pi}(L)\nu(L).$$
Multiplying through by $t^{\alpha-1}$ and integrating over $(0,1)$ yields
$$\mu(K)+\frac{\mu(L)^p-\mu(K)^p}{\alpha p\mu(K)^{p-1}}\leq \frac{\beta}{\alpha} d_{\Pi}(L)\nu(L).$$
Upon re-arranging, we obtain
\begin{equation}
\label{eq:con_hm_1}
\mu(L)^p \leq d_{\Pi}(L)\beta p \mu(K)^{p-1}\nu(L) +(1-\alpha p)\mu(K)^p .
\end{equation}
When $p=1/\alpha$, we obtain
$$\mu(L)^p \leq \frac{\beta d_{\Pi}(L)}{\alpha} \mu(K)^{p-1}\nu(L).$$
Upon re-arranging, we then have
$$\mu(K) \leq \left(\frac{\beta\nu(L)}{\alpha\mu(L)}d_{\Pi}(L)\right)^{\frac{1}{1-p}}\mu(L).$$

Using that $(1-p)=(\alpha-1)/\alpha$ yields the result. For the general case of $p\alpha\leq 1$, we return to Equation~\ref{eq:con_hm_1} and re-arrange to obtain
$$\left(\frac{\mu(L)}{\mu(K)}\right)^{p-1} \leq d_{\Pi}(L)\beta p\frac{\nu(L)}{\mu(L)} +\frac{\mu(K)}{\mu(L)}(1-\alpha p).$$
Since $p\in(0,1/n]$, one has that $p<1$ and we obtain
$$\mu(L) \left[d_{\Pi}(L)\beta p\frac{\nu(L)}{\mu(L)} +\frac{\mu(K)(1-\alpha p)}{\mu(L)}\right]^{\frac{1}{1-p}} \geq \mu(K).$$
\end{proof}

Clearly, the general result in Proposition~\ref{p:q_4} is less than ideal. From this, we are motivated to restrict our measures under consideration. Recall that, if the homogeneity of $\alpha$ is greater than $n$, $\mu$ is also $1/\alpha$-concave from Proposition~\ref{p:gal}. This is a large class of measures, and in some sense the measures that "behave similarly to volume." Using Minkowski's inequality in the form of Equation~\ref{eq:mr} then allows us to state an updated version of Proposition~\ref{p:q_4}.
\begin{lemma}
\label{l:q_3}
Suppose $\mu$ is a $p$-concave and $\alpha$-homogeneous Radon measure, such that either $\alpha \geq n$ and $p>0$ or $\alpha<n$ but $p=\frac{1}{\alpha}$, and $\nu$ is a $\beta$-homogeneous, $\beta>0$, Borel measure. Let $K$ and $L$ be symmetric convex bodies such that
$$h_{\Pi_\mu K}(\theta)\leq h_{\Pi_\nu L}(\theta)$$ for all $\theta\in\s^{n-1}.$ Then, one has that $\mu(K)\leq\mathcal{B}\mu(L)$, where
$$\mathcal{B}=\left(\frac{\beta\nu(L)}{\alpha\mu(L)} d_{\Pi}(L)\right)^{\frac{\alpha}{\alpha-1}}.$$
\end{lemma}
\begin{proof}
From Lemma~\ref{l:q_1} we have that $\frac{1}{\alpha}\mu(K,L) \leq \frac{\beta}{\alpha} d_{\Pi}(L)\nu(L).$ If $\alpha<n$, then we have $\mu$ is $1/\alpha$-concave by hypothesis. If $\alpha\geq n$, then, from Proposition~\ref{p:gal}, $\mu$ is still $1/\alpha$-concave. In either case, one has that $$\frac{1}{\alpha}\mu(K,L)\geq \mu(K)^{1-\frac{1}{\alpha}} \mu(L)^{\frac{1}{\alpha}}.$$
Inserting into the above and re-arranging yields 
$\left(\frac{\mu(K)}{\mu(L)}\right)^{1-\frac{1}{\alpha}} \leq \frac{\beta}{\alpha} d_{\Pi}(L)\frac{\nu(L)}{\mu(L)}.$
Which then yields the result.
\end{proof}
We can now prove Theorem~\ref{t:q_1} from the introduction, which relates $\mu(K)$ and $\nu(L)$ when $\mu$ is $\alpha$ homogeneous and $1/\alpha$-concave.
\begin{proof}[Proof of Theorem~\ref{t:q_1}]
If $\alpha<n$, then we have $\mu$ is $1/\alpha$-concave by hypothesis. If $\alpha\geq n$, then, from Proposition~\ref{p:gal}, $\mu$ is still $1/\alpha$-concave. In either case, from by Lemma~\ref{l:q_3}, one has $\mu(K)\leq \mathcal{B}\mu(L),$ where $\mathcal{B}=\left(\frac{\beta\nu(L)}{\alpha\mu(L)} d_{\Pi}(L)\right)^{\frac{\alpha}{\alpha-1}}.$ If one writes $\mathcal{B}=\left(\frac{\beta\mu(K)}{\alpha\mu(L)}\frac{\nu(L)}{\mu(K)} d_{\Pi}(L)\right)^{\frac{\alpha}{\alpha-1}}$ and re-arranges the resultant inequality to isolate $\frac{\nu(L)}{\mu(K)}$, one obtains that $$\mu(K)\leq\frac{\beta}{\alpha}\mathcal{A}_K d_{\Pi}(L)\nu(L).$$ If one instead writes $\mu(L)=\left(\frac{\nu(L)}{\mu(L)}\right)^{-1}\nu(L)$ on the right-hand side of the inequality, one immediately obtains $\mu(K)\leq\frac{\beta}{\alpha}\mathcal{A}_Ld_{\Pi}(L)\nu(L)$.
\end{proof} When $\mu=\nu$, we obtain the following immediate corollary of Theorem~\ref{t:q_1}.

\begin{corollary}
\label{cor:q_1}
Let $K$ and $L$ be symmetric convex bodies. Furthermore, suppose a Radon measure $\mu$ is $p$-concave and $\alpha$-homogeneous, $\alpha>0$, such that either $\alpha \geq n$ and $p>0$ or $\alpha<n$ but $p=\frac{1}{\alpha},$ and
$$h_{\Pi_\mu K}(\theta)\leq h_{\Pi_\mu L}(\theta)$$ for all $\theta\in\s^{n-1}.$ Then, one has that
$$\mu(K) \leq d^{\frac{\alpha}{\alpha-1}}_{\Pi}(L) \mu(L).$$
Also, if $L$ is a projection body then: $\mu(K) \leq \mu(L).$
\end{corollary}

We now consider the case when $\mu$ is a $F$-concave measure.
\begin{theorem}
\label{t:hard_1}
Let Borel measures $\mu$ and $\nu$ with continuous densities be such that $\mu$ is $\alpha$-homogeneous, $\alpha >1$, and $\nu$ is $\beta$-homogeneous. Suppose that $\mu$ is $F$-concave where $F$ is a differentiable almost everywhere function such that $\frac{1}{zF^\prime(z)}$ and $\frac{F(z)}{zF^\prime(z)}$ are in $L^1_{\text{loc}}([0,\mu(\R^n)))$. If $K$ and $L$ are symmetric convex bodies such that $$h_{\Pi_{\mu} K}(\theta) \le h_{\Pi_{\nu} L}(\theta)$$ for every $\theta\in\s^{n-1},$ then, there exists a quantity $\mathcal{Z}$ and a quantity $\mathcal{L}$ such that 
$\mu(K) \le \mathcal{Z} \nu(L) +\mathcal{L}$, where
\begin{equation}
\mathcal{Z}=\frac{\beta d_{\Pi}(L)}{\alpha-1} \quad \text{and} \quad \mathcal{L}= \frac{1}{\alpha}\int_{0}^{\mu(K)} \frac{F(z)}{zF^\prime(z)}dz -  \frac{F(\mu(L))}{\alpha} \int_{0}^{\mu(K)}\frac{dz}{zF^\prime(z)}.
\label{eq:Z}
\end{equation}
Furthermore, the $\mathcal{Z}$ can be made independent of $L$ if the bound $\mathcal{Z}\leq \frac{\beta \sqrt{n}}{\alpha-1}$ is taken.
\end{theorem}
\begin{proof}
From Lemma~\ref{l:q_1}, we have that
$\mu(K,L) \leq \beta d_{\Pi}(L)\nu(L),$
which then yields
$\mu(tK,L) \leq \beta t^{\alpha-1}d_{\Pi}(L)\nu(L).$
Suppose $\mu$ is $F$-concave such that $F$ is a differentiable almost everywhere function. Then, by Lemma~\ref{t:min_eq} and Equation~\ref{eq:hom_dens}, we get
\begin{equation}
\mu(tK, K)+\frac{F(\mu(L))-F(\mu(tK))}{tF^{\prime}(\mu(tK))}\leq t^{\alpha -2}\beta d_{\Pi} (L)\nu(tL,L).
\label{eq:hard_prob}
\end{equation}
We will now integrate the inequality over $(0,1)$. In particular, we obtain 
$$
\int_{0}^{1}\mu(tK, K)dt = \mu(K) \; \text{ and } \;
\int_0^1 \beta t^{\alpha-2}d_{\Pi}(L)\nu(L)dt=\frac{\beta d_{\Pi}(L)}{\alpha-1}\nu(L).
$$
Notice that
\begin{equation*}
\int_{0}^{1}\frac{F(\mu(L))}{tF^{\prime}(\mu(tK))}dt = \frac{F(\mu(L))}{\alpha} \int_{0}^{\mu(K)}\frac{dz}{zF^\prime(z)},
\end{equation*}
and
\begin{equation*}
\int_{0}^{1}\frac{F(\mu(tK))}{tF^{\prime}(\mu(tK))}dt = \int_{0}^{1}\frac{F(t^{\alpha}\mu(K))}{tF^{\prime}(t^{\alpha}\mu(K))}dt
=\frac{1}{\alpha}\int_{0}^{\mu(K)} \frac{F(z)}{zF^\prime(z)}dz.
\end{equation*} 
Then, Equation~\ref{eq:hard_prob} becomes
$$
\mu(K) \leq \frac{\beta d_{\Pi}(L)}{\alpha-1}\nu(L) + \frac{1}{\alpha}\int_{0}^{\mu(K)} \frac{F(z)}{zF^\prime(z)}dz - \frac{F(\mu(L))}{\alpha} \int_{0}^{\mu(K)}\frac{dz}{zF^\prime(z)},
$$
and the claim follows.
\end{proof}

By applying this result to the particular case of $p$-concave measures, we obtain an analog of Proposition~\ref{p:q_4} relating $\mu(K)$ and $\nu(L)$ that reduces to the result of Theorem~\ref{t:q_1} in the case when $\alpha p=1$. This shows how $\mathcal{L}$ contributes to the sharpness of the inequality in Theorem~\ref{t:hard_1} and in general cannot be removed.
\begin{corollary}
\label{cor:hard_2}
Let Borel measures $\mu$ and $\nu$ with continuous densities be such that $\mu$ is $\alpha$-homogeneous, $\alpha>1$, and $p$-concave and $\nu$ is $\beta$-homogeneous. Suppose $K$ and $L$ are symmetric convex bodies such that $$h_{\Pi_{\mu} K}(\theta) \le h_{\Pi_{\nu} L}(\theta)$$ for every $\theta\in\s^{n-1}.$ Then one has that
$\mu(K) \le \mathcal{Z} \nu(L)$
where
$$\mathcal{Z}= \left(\frac{\nu(L)}{\mu(L)}\right)^\frac{p}{1-p}\left(\frac{\alpha p\beta(1-p) d_{\Pi}(L)}{\alpha-1}+\frac{(1-\alpha p)(1-p)\mu(K)}{\nu(L)}\right)^{\frac{1}{1-p}},$$
and, specifically, when $\alpha p=1$ one has
$$\mathcal{Z}=\left(\frac{\nu(L)}{\mu(L)}\right)^\frac{1}{\alpha-1}\left(\frac{\beta }{\alpha}d_{\Pi}(L)\right)^{\frac{\alpha}{\alpha-1}}.$$
\end{corollary}
\begin{proof}
In Theorem~\ref{t:hard_1} set $F(z)=z^p$; we obtain that
$$\mathcal{L}=\frac{1}{\alpha p}\left(\mu(K)-\left(\frac{\mu(L)}{\mu(K)}\right)^p\frac{\mu(K)}{1-p}\right).$$
Multiplying the resultant inequality through by $\alpha p$ yields
$$\alpha p\mu(K)\leq \frac{\alpha p\beta d_{\Pi}(L)}{\alpha-1}\nu(L)+\mu(K)-\left(\frac{\mu(L)}{\mu(K)}\right)^p\frac{\mu(K)}{1-p}.$$
Upon re-arrangement, we obtain
$$\mu(K)^{1-p}\leq \frac{1-p}{\mu(L)^p}\left(\frac{\alpha p\beta d_{\Pi}(L)}{\alpha-1}\nu(L)+(1-\alpha p)\mu(K)\right).$$
Writing $1=\nu(L)^{1-p}\nu(L)^{p-1}$ yields
$$\mu(K)^{1-p}\leq \nu(L)^{1-p}\left(\frac{\nu(L)}{\mu(L)}\right)^p\left(\frac{\alpha p\beta(1-p) d_{\Pi}(L)}{\alpha-1}+\frac{(1-\alpha p)(1-p)\mu(K)}{\nu(L)}\right).$$
Taking the $(1-p)$th root of both sides then yields
$$\mu(K)\leq \nu(L)\left(\frac{\nu(L)}{\mu(L)}\right)^\frac{p}{1-p}\left(\frac{\alpha p\beta(1-p) d_{\Pi}(L)}{\alpha-1}+\frac{(1-\alpha p)(1-p)\mu(K)}{\nu(L)}\right)^{\frac{1}{1-p}},$$
as claimed. Furthermore, suppose $\alpha p=1$. Then, we see that $(1-p)=(\alpha-1)/\alpha$ and we obtain
$$\mu(K)\leq \nu(L)\left(\frac{\nu(L)}{\mu(L)}\right)^\frac{1}{\alpha-1}\left(\frac{\beta }{\alpha}d_{\Pi}(L)\right)^{\frac{\alpha}{\alpha-1}}.$$
\end{proof}

We would like to remove the quantities $\frac{\nu(L)}{\mu(L)}$, $\frac{\mu(K)}{\mu(L)}$ and $\frac{\mu(K)}{\nu(L)}$ in the above results.  To avoid a preponderance of corollaries, we merely outline these in the following remark. We remind the reader that a symmetric convex body $K$ is said to be in John position if its ellipsoid of maximal volume is the unit ball, i.e. $B_2^n\subset K\subset \sqrt{n} B_2^n$.

\begin{remark}
The results of Proposition \ref{p:q_4}, Lemma~\ref{l:q_3}, Theorems~\ref{t:q_1} and \ref{t:hard_1}, and Corollary~\ref{cor:hard_2} can be made independent of the quantities $\frac{\nu(L)}{\mu(L)}$, $\frac{\mu(K)}{\mu(L)}$ and $\frac{\mu(K)}{\nu(L)}$ by making the following estimates:
\begin{enumerate}
\item If $L$ is in John position, then $\frac{\nu(L)}{\mu(L)} \leq n^\frac{\beta}{2} \frac{\nu(B_2^n)}{\mu(B_2^n)}$.
    \item If $K$ and $L$ are both in John position, one can take $\frac{\mu(K)}{\nu(L)}\leq n^{\frac{\alpha}{2}}\frac{\mu(B_2^n)}{\nu(B_2^n)}$ and $\frac{\mu(K)}{\mu(L)}\leq n^{\frac{\alpha}{2}}.$
\end{enumerate}
\end{remark}

Furthermore, the question does not make sense for log-concave measures, e.g. the Gaussian measure, as one can find bodies such that $h_{\Pi_{\gamma_n}K}(\theta)\geq h_{\Pi_{\gamma_n}L}(\theta)$ and yet $\gamma_n(K)\le \gamma_n(L),$ even in the case when $K$ and $L$ are projection bodies (in fact, one can consider dilates of the Euclidean ball). We conclude this section with stability results for the above isomorphic Shephard Problem, focusing on the case where $\nu$ is $\beta$-homogeneous, $\beta>0$ and $\mu$ is $\alpha\in(0,1)\cup(1,\infty)$-homogeneous and $1/\alpha$-concave. We will use a technique inspired by the stability results in \cite{AK05}.
\begin{theorem}[Stability in Isomorphic Shephard]
Let Borel measures $\mu,\nu$ with continuous densities be such that $\nu$ is $\beta$-homogeneous, $\beta>0$ and $\mu$ is $\alpha\in(0,1)\cup(1,\infty)$-homogeneous and $1/\alpha$-concave. Fix $\epsilon<<1$ and consider symmetric $K,L\in\conbod_0$ such that
$$h_{\Pi_{\mu} K}(\theta)\leq h_{\Pi_{\nu} L}(\theta)-\epsilon.$$
Then,
$$\mu(K)^{1-\frac{1}{\alpha}}\leq \frac{\beta}{\alpha}\frac{\nu(L)}{\mu(L)}d_{\Pi}(L)\mu(L)^{1-\frac{1}{\alpha}}-C(\mu)\epsilon,$$
where $C(\mu)$ depends on $\mu$ only. In particular, if $\mu=\nu,$ then
$$\mu(K)^{1-\frac{1}{\alpha}}\leq d_{\Pi}(L)\mu(L)^{1-\frac{1}{\alpha}}-C(\mu)\epsilon.$$
\end{theorem}
\begin{proof}
By approximation, we may assume $K$ and $L$ are of class $C^2_+$. Start by writing that $\epsilon=\frac{1}{\kappa_{n-1}}\ct{\epsilon}{\theta}$, and therefore we can write our hypothesis as
$$\ct{\left(\phi\circ n_K^{-1}\right)f_K}{\theta} \leq \ct{\left(\psi\circ n_L^{-1}\right)f_L}{\theta}-\frac{1}{\kappa_{n-1}}\ct{\epsilon}{\theta}.$$
We can again treat $\phi\circ n^{-1}_K$ and $\psi\circ n^{-1}_L$ as even functions.

Approximating $L$ by a projection body $M$ in the sense that $M\subset L \subset d_{\Pi}(L) M,$ and then integrating by the Borel measure associated with $h_M$ from Lemma~\ref{l:proj_meas} yields
\begin{align*}\int_{\s^{n-1}}h_M(u)\phi\left(n_{K}^{-1}(u)\right) f_{K}(u) du \leq \int_{\s^{n-1}} &h_M(u) \psi\left(n_L^{-1}(u)\right)f_L (u)du 
\\
&- \frac{\epsilon}{\kappa_{n-1}}\int_{\s^{n-1}}h_M(u)du.\end{align*}
Using the conditions on $M$, we have
\begin{align*}\frac{1}{d_{\Pi}(L)}\int_{\s^{n-1}}h_L(u)\phi\left(n_{K}^{-1}(u)\right) f_{K}(u) du \leq \int_{\s^{n-1}} &h_L(u) \psi\left( n_L^{-1}(u)\right)f_L (u)du
\\
&-\frac{\epsilon}{\kappa_{n-1}d_{\Pi}(L)}\int_{\s^{n-1}}h_L(u).
\end{align*}
Upon re-arrangement, this becomes
$$\mu(K,L)\leq d_{\Pi}(L)\nu(L,L)-\frac{\epsilon}{\kappa_{n-1}}\int_{\s^{n-1}}h_L(u)du.$$
Using the homogeneity of $\nu$, we have $\nu(L,L)=\beta\nu(L).$ We can also use Minkowski's inequality for $\alpha$-homogeneous, $1/\alpha$-concave measures in the form Equation~\ref{eq:mr} to obtain
$$\alpha\mu(K)^{1-\frac{1}{\alpha}} \mu(L)^{\frac{1}{\alpha}}\leq \beta d_{\Pi}(L)\nu(L)-\frac{\epsilon}{\kappa_{n-1}}\int_{\s^{n-1}}h_L(u)du.$$
Upon re-arrangement, this can be written as
$$\mu(K)^{1-\frac{1}{\alpha}} \leq \frac{\beta}{\alpha}\frac{\nu(L)}{\mu(L)} d_{\Pi}(L)\mu(L)^{1-\frac{1}{\alpha}}-\frac{1}{\mu(L)^{1/\alpha}}\frac{\epsilon}{\alpha\kappa_{n-1}}\int_{\s^{n-1}}h_L(u)du.$$

Next, let $P=\s^{n-1}\cap \text{supp}(\phi)$. Notice that from the definition of $P$, it depends on $\mu$ only. We can then write that
\begin{align*}-\frac{1}{\mu(L)^{1/\alpha}}\frac{\epsilon}{\alpha\kappa_{n-1}}\int_{\s^{n-1}}h_L(u)&\leq -\frac{1}{\mu(L)^{1/\alpha}}\frac{\epsilon}{\alpha\kappa_{n-1}}\int_{P}h_L(u)
\\
&= -\frac{1}{\mu(L)^{1/\alpha}}\frac{\epsilon}{\alpha\kappa_{n-1}}\int_{\s^{n-1}}h_L(u)\chi_P(u).
\end{align*}

By Theorem~\ref{t:min_exist_uni}, there exists a symmetric convex body $R$ such that $\frac{1}{\kappa_{n-1}}\chi_P(u)du=dS^{\mu}_{R}(u).$ Therefore, we obtain, by using again Equation~\ref{eq:mr}, that
$$-\frac{1}{\mu(L)^{1/\alpha}}\frac{\epsilon}{\alpha\kappa_{n-1}}\int_{\s^{n-1}}h_L(u)\leq -\frac{\epsilon}{\alpha\mu(L)^{1/\alpha}}\mu(R,L)\leq -\epsilon \mu(R)^{1-1/\alpha}.$$
Setting $C(\mu)=\mu(R)^{1-1/\alpha}$ yields the result.
\end{proof}

\section{Applications to Capacity}
\label{s:con}
We begin this section by classifying projection bodies as the measure varies and the convex body is fixed. In order to complete this endeavour, we must discuss harmonic functions. All of the following concerning harmonic functions can be found in \cite{Evans}. Let $\Omega$ be a connected, bounded domain. If $\triangle$ denotes the Laplacian operator, we say $u\in C^2(\overline{\Omega})$ is $\textit{harmonic}$ if $\triangle u = 0$ on $\Omega$.
It can be shown that a $C^2$ function $u$ is harmonic if, and only if, $u$ satisfies the \textit{mean value property}:
For every $B(x,r)\subset \Omega,$ (where $B(x,r)$ is the ball of radius $r$ centered at $x$) one has
\begin{equation}
    u(x)=\frac{1}{r^n\kappa_n}\int_{B(x,r)}u(y)dy=\frac{1}{r^{n-1}n\kappa_n}\int_{\partial B(x,r)}u(y)dy.
    \label{eq:mean_value}
\end{equation}
We see from Equation~\ref{eq:mean_value} that the value on the boundary of a set determines the value in the interior of a set for a harmonic function; using this, one can easily show the following maximum principle: suppose $u,v\in C(\overline{\Omega})$ such that $u$ and $v$ are harmonic, if $u\leq v$ on $\partial \Omega$, then $u\leq v$ on $\Omega.$ Next, consider some $f\in C(\partial \Omega)$. Then, the following is \textit{the Dirichlet problem} for $\Omega$ with boundary data $f$:
\begin{equation}
    \begin{cases}
    \triangle u = 0 \; &\text{on } \Omega
    \\
    u=f \; &\text{on } \partial \Omega.
    \end{cases}
    \label{dirichlet}
\end{equation}
If a solution exists, then it is unique; the conditions for existence are technical, but if $\Omega$ is bounded and convex, then the Dirichlet problem is solvable.

\begin{definition}
Fix a symmetric $K\in\conbod_0$ and an even Borel measure $\mu$ with continuous density $\phi$. Then, we define $\mathcal{P}_{K}$ as the set of even Borel measures with continuous densities such that $\nu\in \mathcal{P}_{K} \mu \leftrightarrow \Pi_{\mu} K = \Pi_{\nu}K.$ 
\end{definition}
In the case where $K$ is of class $C^2_+$ and $\phi$ is continuous, one has $\nu\in \mathcal{P}_{K} \mu$ with continuous density $\psi$ if, and only if, $$\ct{\left(\phi\circ n^{-1}_K\right)dS_K}{\theta}=\ct{\left(\psi\circ n^{-1}_K\right)dS_K}{\theta}$$
for all $\theta\in\s^{n-1}.$ Since $K$ is symmetric and $\mu$ and $\nu$ are even, $\psi\circ n^{-1}_K, \phi\circ n^{-1}_K$  and $f_K$ are even functions; we have, from the uniqueness theorem of the cosine transform, Lemma~\ref{l:funk}, that $(\phi\circ n^{-1}_K)(\theta)=(\psi\circ n^{-1}_K)(\theta)$ for almost all $\theta\in\s^{n-1}$. But, identifying $y\in\partial K$ with $n^{-1}_K(\theta)$ one has that $\phi=\psi$ for almost all $y\in\partial K.$ However, equality on the boundary of $K$ does not yield equality on all of $K$, for example $\psi(x)=\phi(x)\|x\|_K.$ If $\nu_1,\nu_2\in \mathcal{P}_{K} \mu$, all we know is that $\nu_1(\partial K)=\nu_2(\partial K)$ and $\diff{\nu_1}{x}=\diff{\nu_2}{x}$ a.e on $\partial K$. If we further suppose that the densities of $\nu_1$ and $\nu_2$ are harmonic on $\text{int}(K)$, then this equality extends to all of $K$ via the Dirichlet problem, Equation~\ref{dirichlet}. We collect these results in the following theorem, where we see that harmonic measures play the role of the Blaschke body.
\begin{proposition}
Let a symmetric convex body $K$ be of class $C^2_+$ and consider an even Borel measure $\mu$ with density that is continuous on $\partial K$. Then there exists a unique even Borel measure $\nu\in \mathcal{P}_{K} \mu$ such that the density of $\nu$ is harmonic on $K$.
\end{proposition}

We next introduce capacity; the following can be found in \cite{EvansGar}. The $p$-capacity of a bounded domain is defined as
\begin{equation}
\capa{p}{\Omega}=\inf \left\{\int_{\R^{n}}|\nabla v|^{p} d x: v \in C_{c}^{\infty}\left(\R^{n}\right), v(x) \geq 1 \; \forall \; x \in \Omega\right\}
 \end{equation}
 where $C_{c}^{\infty}\left(\R^{n}\right)$ is the set of functions from $C^{\infty}\left(\R^{n}\right)$ with compact support. In the case of $p=2$, $\capa{2}{\Omega}$ is the Newton capacity of $\Omega.$
 The $p$-capacity of a set is related to the $p$-Dirichlet problem:
 suppose $u$ decays at infinity, i.e. $\lim_{|x|\to\infty}u(x)=0$, and $\Omega$ is a bounded domain. Then, the $p$-Dirichlet problem is the following system:
 \begin{equation}
 \begin{cases}
     \text{div}(|\nabla u|^{p-2}\nabla u) =0, \; &\text{in }\R^n\setminus\Omega
     \\
     u(x) = 1 \; &\text{on } \partial \Omega
\end{cases}
 \end{equation}
  The solution to the $p$-Dirichlet problem is called the \textit{p-capacitary function} and satisfies $0 < u <1$ and $|\nabla u| \neq 0 $ in $\R^n\setminus \overline{\Omega}.$ 
The solution to the $p$-Dirichlet problem generates a Borel measure on $\s^{n-1}$: let $E$ be a Borel subset of $\s^{n-1}$. Then, the $p$-capacitary measure of $\Omega$ is given by \cite{CJL96, CNSXYZ15, CS03, DJ96_2}:
 $$\mu_p(\Omega,E)=(p-1)\int_{n_\Omega^{-1}(E)}|\nabla u(y)|^p dy.$$
The Poincar{\'e} formula for $p$-capacity for $1<p<n$ is given by \cite{CNSXYZ15}:
 \begin{equation*}
 \begin{split}
     C_p(\Omega)&=\frac{1}{n-p}\int_{\s^{n-1}}h_{\Omega}(\xi)d\mu_p(\Omega,\xi)
     \\
     &=\frac{p-1}{n-p}\int_{\s^{n-1}}h_{\Omega}(\xi)|\nabla u(n_\Omega^{-1}(\xi))|^p dS_\Omega(\xi),
     \end{split}
 \end{equation*}
 where the second equality holds if $\Omega$ is a strictly convex domain. And it is truly amazing that this is essentially the same formula for volume. In fact, if $\overline{\Omega}$ is a convex body, then the above is just $\frac{1}{n-p}\tilde{\mu}(\overline{\Omega},\overline{\Omega})$ where $\tilde{\mu}$ is the measure whose density is given by $(p-1)|\nabla u(x)|^p$. There exists an analogue of mixed volume for capacity. Let $\Omega_{0}, \Omega_{1}$ be bounded convex domains in $\R^{n}$. The mixed $p$-capacity of $\Omega_{0}$ and $\Omega_{1}$ is given by
$$
C_{p}\left(\Omega_{0}, \Omega_{1}\right)=\frac{1}{n-p} \int_{\s^{n-1}} h_{\Omega_{1}}(\xi) d \mu_{p}\left(\Omega_{0}, \xi\right)
$$
One has, $C_{p}\left(\Omega_{0}, \Omega_{0}\right)=C_{p}\left(\Omega_{0}\right) .$ The following analogue of Minkowski's inequality for $p$ -capacity in the context of bounded convex domains was shown in \cite{CNSXYZ15}:

\begin{lemma}[Minkowski's Inequality for Capacity]
\label{l:min_cap}
Suppose $1<p<n .$ Let $\Omega_{0}, \Omega_{1}$ be bounded convex domains in $\R^{n}$. Then,
$$
C_{p}\left(\Omega_{0}, \Omega_{1}\right)^{n-p} \geq C_{p}\left(\Omega_{0}\right)^{n-p-1} C_{p}\left(\Omega_{1}\right)
$$
with equality if and only if $\Omega_{0}, \Omega_{1}$ are homothetic.
\end{lemma}

  \begin{question}
 \label{q:cap}
 Let $p,q\in (1,n)$. Consider symmetric convex bodies $K$ and $L$. Let $u$ be the $p$-harmonic solution to the $p$-Dirichlet problem for $\textnormal{int}(K)$ and $v$ be the $q$-harmonic solution to the $q$-Dirichlet problem for $\textnormal{int}(L)$. Let $\mu$ be the measure with density $|\nabla u|^p$ and $\nu$ be the measure with density $|\nabla v|^q$. If
 $$h_{\Pi_\mu K} (\theta)\leq h_{\Pi_\nu L} (\theta),$$
 does there exist a quantity $\mathcal{M}>0$ such that
 $C_p(\textnormal{int} (K))\leq \mathcal{M}C_p(\textnormal{int}(L))?$
 \end{question}

 \begin{theorem}
 Under the notation and assumptions of Question~\ref{q:cap}, one has, if $p<n-1$:
\begin{equation*}
    C_p(\text{int} (K)) \leq\left(\frac{p-1}{n-p}\frac{n-q}{q-1}d_{\Pi}(L)\right)^{\frac{n-p}{n-p-1}} \left(\frac{C_q(\text{int}(L))}{C_p(\text{int}(L))}\right)^\frac{1}{n-p-1}C_q(\text{int}(L)),
\end{equation*}
and if $p\in(n-1,n)$:
\begin{equation*}
    C_q(\text{int}(L)) \leq \left(\frac{p-1}{n-p}\frac{n-q}{q-1}d_{\Pi}(L)\right)^{\frac{n-p}{p-(n-1)}} \left(\frac{C_q(\text{int}(L))}{C_p(\text{int}(L))}\right)^\frac{1}{p-(n-1)} C_p(\text{int} (K)).
\end{equation*}
In particular, if $p=q$ the constants simplify to $d_{\Pi}(L)$ raised to a power. If additionally $L$ is a projection body, then we can conclude
$C_p(\text{int} (K)) \leq C_p(\text{int} (L))$ if $p\in(1,n-1)$ and, surprisingly, $C_p(\text{int} (L)) \leq C_p(\text{int} (K))$ if $p\in (n-1,n)$.
 \end{theorem}
 \begin{proof}
Via approximation, one can suppose that $K$ and $L$ are $C^2_+$. From Lemma~\ref{l:q_0} we can state:
\begin{equation*}
\begin{split}
    \int_{\s^{n-1}}h_L(u)&|\nabla u\left(n_K^{-1}(u)\right)|^p f_K(u) du 
    \\
    &\leq d_{\Pi}(L)\int_{\s^{n-1}} h_L(u) |\nabla v\left(n_L^{-1}(u)\right)|^q f_L(u)du.
    \end{split}
\end{equation*}
Writing this in terms of mixed capacity, this is precisely
\begin{equation*}
    C_p(\text{int} (K),\text{int} (L)) \leq\frac{p-1}{n-p}\frac{n-q}{q-1} d_{\Pi}(L)C_q(\text{int}(L)).
\end{equation*}
Raising each side to the $n-p$ then yields
\begin{equation*}
    C_p(\text{int} (K),\text{int} (L))^{n-p} \leq\left(\frac{p-1}{n-p}\frac{n-q}{q-1}\right)^{n-p} d^{n-p}_{\Pi}(L)C_q(\text{int}(L))^{n-p}.
\end{equation*}
Applying Lemma~\ref{l:min_cap} and raising both sides to appropriate powers then yields the claim.
\end{proof}

\begin{remark}The case $p=n-1$ does not follow from the above, as in this instance Lemma~\ref{l:min_cap} reads
$ C_p(\text{int} (K),\text{int} (L))\geq  C_p(\text{int} (L))$, and so the mixed capacity seemingly does not carry enough information about $K$.
\end{remark}

\section{Appendix: $L^q$-Measure Surface Area Problem}
In this section, we repeat much of Section~\ref{s:min}, with the goal of defining $L^q$ mixed measures and solving the weighted $L^q$ Minkowski problem for measures in $\Lambda^n$. The standard notation is to refer to everything that follows as $L^p$; however, to avoid confusion with $p$-concavity, we will use the notation $L^q$. We recall that Firey's $L^q$-Minkowski Addition is the following: for $K,L\in\conbod$,$\;\alpha,\beta\in\R$ and $q\geq 1$, one has that $\alpha\cdot_q K+_q \beta\cdot_q L=\alpha\cdot K+_q \beta\cdot L$ is the convex body whose support function is given by
$$h^q_{\alpha\cdot K+_q \beta\cdot L}(x)=\alpha h_K^q(x)+\beta h_L^q(x),$$
where the first notation emphasises the fact that $K$ and $L$ are dilated by $\alpha^{1/q}$ and $\beta^{1/q}$ respectively; we will use the second notation, which suppresses this fact.
The $L^q$ surface area measure of $K$ is the Borel measure on the sphere that satisfies the following limit
$$\lim_{\epsilon\to 0}\frac{\vol_n\left(K+_q\epsilon\cdot B_2^n\right)-\vol_n(K)}{\epsilon}=\frac{1}{q}\int_{\s^{n-1}}dS_{K,q}(u).$$
As it turns out, one has that \cite{Sh1},
$$dS_{K,q}(u)=h_K^{1-q}(u)dS_{K}(u).$$
We first obtain the following set inclusion.
\begin{proposition}
\label{p:q_set}
Suppose $q\geq q^\prime\geq 1$. Then,
$$(1-\lambda)\cdot K+_q \lambda\cdot L \supseteq (1-\lambda)\cdot K+_{q^\prime} \lambda\cdot L \supseteq (1-\lambda) K+ \lambda L.$$
\end{proposition}
\begin{proof}
Write
\begin{align*}h^{q^\prime}_{(1-\lambda)\cdot K+_q \lambda\cdot L}(x)&=\left((1-\lambda) h_K^q(x)+\lambda h_L^q(x)\right)^{\frac{q^\prime}{q}}
\\
&=\left((1-\lambda) (h_K^{q^\prime}(x))^{\frac{q}{q^\prime}}+\lambda (h_L^{q^\prime}(x))^{\frac{q}{q^\prime}}\right)^{\frac{q^\prime}{q}}.
\end{align*}
Since $1\geq \frac{q^\prime}{q},$ we obtain from Jensen's inequality that $$h^{q^\prime}_{(1-\lambda)\cdot K+_q \lambda\cdot L}(x)\geq h^{q^\prime}_{(1-\lambda)\cdot K+_{q^\prime} \lambda\cdot L}(x),$$ and the claim follows.
\end{proof}
Next, suppose a Borel measure $\mu$ with density is $F$-concave, that is it satisfies Equation~\ref{eq:concave} for some invertible function $F$. Using Proposition~\ref{p:q_set}, we obtain the following result.
\begin{proposition}
\label{p:qcon}
Suppose $\mu$ is $F$-concave with respect to Minkowski addition on a class of compact sets $\mathcal{C}$. Then, for $q\geq 1$, $\mu$ is also $F$-concave with respect to $L^q$-Minkowski addition, that is for $K,L\in\mathcal{C}$ and $\lambda\in[0,1]$, one has
\begin{equation}
\label{eq:fcon_q}
\mu\left((1-\lambda)\cdot K+_q\lambda \cdot L\right)\geq F^{-1}\left((1-\lambda) F(\mu(K)) +\lambda F(\mu(L))\right),
\end{equation}
and equality implies equality in Equation~\ref{eq:concave}.
\end{proposition}
\noindent We will now go over an important example using Proposition~\ref{p:q_set}, the Gaussian measure:
\begin{example}
Let $K,L$ be Borel sets of $\R^n$, $\lambda\in(0,1),$ and $q\geq 1$. Then, the Gaussian measure $\gamma_n$ satisfies the following $L^q$-concavities:
\begin{enumerate}
    \item Ehrhard inequality: \cite{Bor03,EHR1,EHR2,Lat96}:$$\Phi^{-1}\left(\gamma_{n}((1-\lambda)\cdot K+_q\lambda\cdot L)\right) \geq(1-\lambda) \Phi^{-1}\left(\gamma_{n}(K)\right)+\lambda \Phi^{-1}\left(\gamma_{n}(L)\right),$$
    with equality if, and only if, $K=L$ or $K$ and $L$ are half-spaces, one nested in the other.
    \item Log-concavity \cite{Bor75, BL76,SD77}: If $K$ and $L$ are closed convex sets with non-empty interior, $$\gamma_{n}((1-\lambda)\cdot K+_q\lambda \cdot L) \geq \gamma_{n}(K)^{1-\lambda} \gamma_{n}(L)^{\lambda},$$
    with equality if, and only if, $K=L$.
    \item Gardner-Zvavitch inequality \cite{EM21,GZ10,PT13,KL21}: If $K$ and $L$ are symmetric convex bodies: $$\gamma_n\left((1-\lambda)\cdot K +_q \lambda \cdot L\right)^{1/n}\geq (1-\lambda)\gamma_n(K)^{1/n} + \lambda \gamma_n(L)^{1/n},$$ with equality if, and only if, $K=L$.
\end{enumerate}
\end{example}
\noindent It may happen that a measure is $F$-concave with respect to some $F$ in the case of $L^q$ summation, for some fixed $q>1$, but not $F$-concave (with respect to the same $F$) for normal Minkowski summation. The prototypical example is the case of volume, shown by Firey \cite{Firey62}, with $F(x)=x^{\frac{q}{n}}$. This example shows that by replacing the addition, the concavity can improve.
\begin{example}
    Let $K,L\in\conbod_0$ and $q >1$. Then, for every $\lambda \in (0,1)$:
    $$\vol_n((1-\lambda)\cdot K +_q \lambda \cdot L)^\frac{q}{n} \geq (1-\lambda)\vol_n(K)^\frac{q}{n} + \lambda \vol_n(L)^\frac{q}{n},$$
    with equality if, and only if, $K$ and $L$ are dilates.
\end{example}

For a Borel measure $\mu$ with density $\phi$ and $K\in\conbod$ such that $\partial K$ is in the Lebesgue set of $\phi$, we define the \textit{$L^q$ weighted surface area measure} as
\begin{equation}
    \label{q_mu_sur}
    S^{\mu}_{K,q}(E)=\int_{n_K^{-1}(E)\cap \partial^\prime K}\phi\left(x\right)h_K^{1-q}(n_K(y))dy,
\end{equation}
for every Borel $E\subset \s^{n-1}$. We now work towards proving the following:
\begin{theorem}
\label{t:third_q}
Fix $q\geq 1$. Let $\mu$ be a Borel measure on $\R^n$ with even, continuous density such that $\mu \in \Lambda^n$ for some $\beta>0.$ Suppose $\nu$ is a finite, even Borel measure on $\s^{n-1}$ that is not concentrated on any hemisphere.
Then, there exists a symmetric convex body $K$ such that
$$d\nu(u)=c_{\mu,K,q} dS^{\mu}_{K,q}(u); \quad c_{\mu,K,q}:=\frac{1}{q}\mu(K)^{\frac{\beta}{n}-1}.$$
\end{theorem}

\noindent Theorem~\ref{t:third_q} in particular solves the following Monge-Amp\`ere equation: consider a given smooth, positive, even data function $f$ on the sphere, $q\geq 1$, and a Borel measure $\mu\in\Lambda^n$ with even, continuous density $\phi$. Then, there exists $\beta>0$ and a symmetric, smooth convex body $K$ such that $h:=h_K$ and $c=c_{\mu,K,q}=\frac{1}{q}\mu(K)^{\frac{\beta}{n}-1}$ solve
$$c\phi(\nabla h)h^{1-q}\det (D^2 h +h I)=f.$$

Liu \cite{JL22} previously considered Theorem~\ref{t:third_q} in the specific case of the Gaussian measure. Most of the proofs of what follows are similar to that in Section~\ref{s:min}, and so we only outline the differences here. Like before, we first need the following lemma.
\begin{lemma}[$L^q$-First Variation of Radial Functions of Wulff Shapes]
	\label{l:first_q}
	Fix $q\geq 1$. Consider a convex body $K\in\conbod_0$ and $f\in C(\s^{n-1})$. Then, there exists is a small $\delta >0$ such that $$h_t(u)=\left(h^q_K(u)+tf(u)\right)^{1/q}$$ is positive for all $u\in\s^{n-1}$ and $|t|< \delta $, implying that, for almost all $u\in\s^{n-1}$ up to a set of spherical Lebesgue measure zero,
	$$\diff{\rho_{[h_t]}(u)}{t}\bigg|_{t=0}=\lim_{t\to 0}\frac{\rho_{[h_t]}(u)-\rho_{K}(u)}{t}=\frac{1}{q}\frac{f(n_{K}(r_{K}(u)))}{h^q_K(n_K(r_K(u)))}\rho_K(u),$$
	and, furthermore, one has that there exists some $M>0$ such that $$|\rho_{[h_t]}(u)-\rho_{K}(u)|<M|t|$$ for all $|t|<\delta.$
	\end{lemma}
	\begin{proof}
	If one takes the Taylor series expansion of $\log(h_t),$ $\log(h_t)=\log(h_K)+\frac{t}{q}\frac{f}{h^q_K}+o(t),$ one obtains that $h_t$ is a logarithmic family of the pair $(h_K,\frac{1}{q}\frac{f}{h^q_K})$, and the claim follows from \cite[Lemma 4.3]{HLYZ16}.
	\end{proof}

	\begin{lemma}[Aleksandrov's $L^q$-Variational Formula For Arbitrary Measures]
	\label{l:second_q}
	Let $\mu$ be a Borel measure on $\R^n$ with locally integrable density $\phi$. Let $K$ be a convex body containing the origin in its interior, such that $\partial K$, up to set of $(n-1)$-dimensional Hausdorff measure zero, is in the Lebesgue set of $\phi$. Then, for a continuous function $f$ on $\s^{n-1}$ and $q\geq 1$, one has that
	$$\lim_{t\rightarrow 0}\frac{\mu([\left(h^q_K+tf\right)^{\frac{1}{q}}])-\mu(K)}{t}=\frac{1}{q}\int_{\s^{n-1}}f(u)dS^{\mu}_{K,q}(u).$$
	\end{lemma}
	\begin{proof}
	Set $h^q_t=h^q_K+tf$ and use polar coordinates to write
	$$\mu([h_t])=\int_{\s^{n-1}}\int_0^{\rho_{[h_t]}(u)}\phi(ru)r^{n-1}drdu=\int_{\s^{n-1}}F_t(u)du,$$
	where
	$F_t(u)=\int_0^{\rho_{[h_t]}(u)}\phi(ru)r^{n-1}dr.$ One has again that $\diff{F_t}{t}\bigg|_{t=0}=F^\prime_0$ exists by the Lebesgue differentiation theorem and also
	\begin{align*}\lim_{t\to 0}\frac{F_t(u)-F_0(u)}{t}=\phi(r_K(u))\frac{f(n_{K}(r_{K}(u)))}{h^q_K(n_K(r_K(u)))}\frac{\rho^{n}_K(u)}{q}.\end{align*} for almost all $u\in \s^{n-1}$ via Lemma~\ref{l:first_q}. Using dominated convergence to differentiate underneath the integral sign yields the result.
	\end{proof}
 \noindent We remark that Wu \cite{DW17} had previously established Lemma~\ref{l:second_q} in the special case when $\mu$ was $p$-concave, $1/p$-homogeneous with continuous density. Liu \cite{JL22} also previously established Lemma~\ref{l:second_q} in the case of the Gaussian measure.  
	\begin{proof}[Proof of Theorem~\ref{t:third_q}]
	The proof is essentially the same as in Theorem~\ref{t:third}. Suppose we have $\beta>0$ so that $\mu\in\Lambda^n$, and $\nu$ is an even Borel measure on the sphere not concentrated on any half-space. Next, define the following functional on the set of even continuous functions on the sphere
	\begin{equation}
	    \label{eq:q_functional}
	    \Psi_{\nu,\beta}(f)=\frac{n}{\beta}\mu\left([f]\right)^\frac{\beta}{n}-\int_{\s^{n-1}}f^q(u)d\nu(u).
	\end{equation}
	One can easily check that $\Psi_{\nu,\beta}(f)\leq \Psi_{\nu,\beta}(h_{[f]})$, and so we have again that the maximiser is the support function of a symmetric convex body. We now check that a maximiser satisfies the claim: consider $$h_t(u)=\left(h^q_{K_\beta}(u)+tf(u)\right)^{1/q},$$ where $K_\beta$ is the symmetric convex body whose support function is a maximiser of $\Psi_{\nu,\beta}$, $f$ is an arbitrary, even function in $f\in C(\s^{n-1}),$ and $\delta$ is picked small enough so that, for $|t|<\delta$, $h_t$ is positive for all $u\in\s^{n-1}$.
	
Viewing $\Psi_{\nu,\beta}(h_t)$ as a function of one variable in $t$, we must have that $\Psi_{\nu,\beta}(h_t)$ has an extremum at $t=0$, and so
\begin{align*}0=\diff{\Psi_{\nu,\beta}(h_t)}{t}\bigg|_{t=0}&=\mu(K_\beta)^{\frac{\beta}{n}-1}\diff{}{t}\left(\mu([h_t])\right)_{t=0}-\diff{}{t}\int_{\s^{n-1}}h^q_{t}(u)d\nu(u)\bigg|_{t=0}
\\
&=\frac{\mu(K_\beta)^{\frac{\beta}{n}-1}}{q}\int_{\s^{n-1}}f(u)dS^{\mu}_{K_{\beta},q}(u)-\int_{\s^{n-1}}f(u)d\nu(u),\end{align*}
where the second equality follows from Lemma~\ref{l:second_q}.
Re-arranging, we then see
$$\int_{\s^{n-1}}f(u)d\nu(u)=\int_{\s^{n-1}}f(u)\left[\frac{\mu(K_\beta)^{\frac{\beta}{n}-1}}{q}dS^{\mu}_{K_{\beta},q}(u)\right].$$
Since continuous functions are dense in $L^1(\s^{n-1})$, the above is true for all even $f\in L^1(\s^{n-1})$. We then have via the Riesz Representation theorem that
$$d\nu(u)=\frac{1}{q}\mu(K_\beta)^{\frac{\beta}{n}-1}dS^{\mu}_{K_{\beta},q}(u),$$
as required. The rest of the proof is the same, up to minor adjustments due to the differences between Equations \ref{eq:functional} and \ref{eq:q_functional}.
\end{proof}
\begin{corollary}
\label{cor:almost_hom_q}
Fix $q\geq 1$. Let $\mu$ be an even, $\alpha$-homogeneous, $\alpha > 0,$ $\alpha\neq q$ Borel measure with density. Suppose $\nu$ is a finite, even measure on $\s^{n-1}$ not concentrated on any hemisphere. Then, there exists a symmetric $K\in\conbod_0$ such that
$$d\nu=dS^{\mu}_{K,q}.$$
\end{corollary}
\begin{proof}
The proof is essentially the same as Lemma~\ref{l:almost_hom}, except we remark that $dS^{\mu}_{L,q}(u)$ is $(\alpha-q)$-homogeneous (in the argument $L$) for $L\in\conbod_0$. Indeed, for $L$ of class $C^2_+$, one has that $dS_{tL,q}(u)=h_{tL}^{1-q}(u)dS_{tL}(u)=t^{n-q}h_L^{1-q}(u)dS_L(u)=t^{n-q}dS_{L,q}(u)$ for $t>0$.  Thus, we obtain, using that $n^{-1}_{tL}(u)=tn^{-1}_{L}(u)$ and the $(\alpha-n)$-homogeneity of $\phi$,
\begin{align*}dS^{\mu}_{tL,q}(u)&=\phi\left(n^{-1}_{tL}(u)\right)dS_{tL,q}(u)
\\
&=t^{n-q}t^{\alpha-n}\phi\left(n^{-1}_{L}(u)\right)dS_{L,q}(u)=t^{\alpha-q}dS^{\mu}_{L,q}(u).\end{align*}
Thus, in the proof of Lemma~\ref{l:almost_hom}, we instead let $A$ be the solution to
$$A^{\alpha-q}=\frac{1}{q}\mu(\tilde{K})^{\frac{\beta}{n}-1},$$
and the rest of the proof is the same.
\end{proof}

From Lemma~\ref{l:second_q}, we can define the following; this had been previously done by Wu \cite{DW17} when the Borel measure $\mu$ had continuous density.
	\begin{definition}
	Let $\mu$ be a Borel measure on $\R^n$ with locally integrable density. Consider $K\in\conbod_0$ such that $\partial K$ is in the Lebesgue set of the density of $\mu$ (up to a set of $(n-1)$-dimensional Hausdorff measure zero). Then, for $L\in\conbod_0$, the \textit{$L^q$-mixed measure of $K$ and $L$} is given by, using Lemma~\ref{l:second_q},
	\begin{align*}
	    \mu_q(K,L):=\liminf_{\epsilon\to 0}\frac{\mu\left(K+_q\epsilon\cdot L\right)-\mu(K)}{\epsilon}
	    =\frac{1}{q}\int_{\s^{n-1}}h_L(u)^qdS^{\mu}_{K,q}(u).
	\end{align*}
	\end{definition}
\noindent Henceforth, we shall assume that the Borel measure $\mu$ has \textit{continuous} density for brevity; we emphasize that all we actually need is that $\partial K$ is in the Lebesgue set of the density of $\mu$.
	
Notice that $\mu_q(K,K)=\frac{1}{q}\mu(K,K)$. Also, if $\mu$ is $\alpha$-homogeneous, then $\mu_q(K,K)=\frac{\alpha}{q}\mu(K)$. Recall that, as a Borel measure with density, we denote the Lebesgue measure as $\lambda$. So, we have that $\vol_n(K)=\frac{q}{n}\lambda_q(K,K)$. We now attempt to show an $L^q$-Minkowski's Inequality for $L^q$-mixed measures. 

\begin{theorem}
\label{t:LqFcon}
Fix $q\geq 1$. Let $\mu$ be a Borel measure on $\R^{n}$, such that $\mu$ is $F$-concave (with respect to $L^q$ summation) and $F$ is differentiable, with respect to a class of Borel sets $\mathcal{C}$. Then, for every $K,L\in\mathcal{C}$, one has that:

$$
\mu_q(K, L) \geq \mu_q(K, K)+\frac{F(\mu(L))-F(\mu(K))}{F^{\prime}(\mu(K))}.
$$
\end{theorem}
\begin{proof}
The proof is same as in Theorem~\ref{t:min_eq} except one sets $K_\lambda=\frac{K}{(1-\lambda)^{1/q}}$ to write $\mu(K+_q\lambda\cdot L)=\mu((1-\lambda)\cdot K_\lambda+_q\lambda \cdot L)$. When taking the derivative, use that $\diff{\mu(K_\lambda)}{\lambda}\bigg|_{\lambda=0}=\frac{1}{q}\mu(K,K)=\mu_q(K,K).$
\end{proof}
	\begin{corollary}[$L^q$-Minkowski's Inequality - Concavity link]
\label{cor:minkowski_q}
Fix $q\geq 1$. Let $K,L\in\conbod_0$ and suppose a Borel measure $\mu$ with continuous density is $F$-concave (with respect to $L^q$ summation) and $F$ is differentiable. Then, $$
\mu_q(K, L) = \mu_q(K, K)+\frac{F(\mu(L))-F(\mu(K))}{F^{\prime}(\mu(K))},
$$
if, and only if, there is equality in Equation~\ref{eq:fcon_q}.
\end{corollary}
\noindent In the volume case, Lutwak \cite{LE93} standardized the normalization $V_q(K,L):=\frac{q}{n}\lambda_q(K,L)$. In this case, Theorem~\ref{t:LqFcon} and Corollary~\ref{cor:minkowski_q} together (with $F(x)=x^\frac{q}{n}$) recover \cite[Theorem 1.2]{LE93}: for $K,L\in\conbod_0$ and $q>1$, one has
$$V_q(K,L)^n \geq \vol_n(K)^{n-q}\vol_n(L)^q,$$
with equality if, and only if, $K$ is a dilate of $L$.

We next list results for $p$-concave measures. We shall use the following, which was shown by Roysdon and Xing \cite{RX21}. This extends the result by Firey in the volume case, i.e. that replacing the summation can improve the concavity.
\begin{proposition}[Brunn-Minkowski inequality for $p$-concave measures with respect to $L^q$ addition]
\label{p:RX}
Let $\mu$ be a Radon measure that is $p$-concave with respect to Minkowski addition, then, for $q\geq 1$, it is $pq$-concave, with respect to $L^q$ Minkowski addition i.e.
\begin{equation}
    \label{eq:pq_concave}
    \mu\left((1-\lambda) \cdot A+_q \lambda \cdot B\right)^{pq} \geq (1-\lambda) \mu(A)^{pq}+\lambda \mu(B)^{pq}.
\end{equation}
If $q>1$, then equality holds if, and only if, $A=aB$ for some $a\in\R^+$.
\end{proposition}

\begin{corollary}
\label{q_min_p}
Fix $q > 1$. Suppose a Radon measure $\mu$ is $p$-concave. Then, for $K,L\in\conbod_0$ in the support of $\mu$, one has from Proposition~\ref{p:q_set}:
$$\mu_q(K,L)\geq \frac{1}{q}\mu(K,K)+\frac{1}{p}\left[\mu(L)^p\mu(K)^{1-p}-\mu(K)\right].$$
If $\mu$ is also $\alpha$-homogeneous, $\alpha>0$, then:
$$\mu_q(K,L)\geq \frac{1}{p}\left[\mu(K)^{1-p}\mu(L)^{p}-\left(1-\frac{p\alpha}{q}\right)\mu(K)\right].$$
\\
We can also use Proposition~\ref{p:RX} to replace $p$ with $qp$ and write:
$$\mu_q(K,L)\geq \frac{1}{q}\mu(K,K)+ \frac{1}{pq}\left[\mu(K)^{1-pq}\mu(L)^{pq}-\mu(K)\right].$$
\\
If $\mu$ is additionally $\alpha$-homogeneous, $\alpha>0$ we can write
$$\mu_q(K,L)\geq  \frac{1}{pq}\left[\mu(K)^{1-pq}\mu(L)^{pq}-(1-p\alpha)\mu(K)\right].$$
\\
Finally, for $\mu$ $\alpha$-homogeneous, $\alpha>0$ that is either $1/\alpha$-concave or $\frac{q}{\alpha}$-concave, one has 
$$\mu^{\alpha}_q(K,L)\geq \left(\frac{\alpha}{q}\right)^{\alpha}\left[\mu(K)^{\alpha-q}\mu(L)^{q}\right].$$
In all inequalities, there is equality if, and only if, $K=aL$ for some $a>0$.
\end{corollary}
Using this, we obtain uniqueness for the $L^q$ weighted surface area measures.
\begin{lemma}
\label{l:q_meas_uni}
Fix $q\geq 1$. Let a Radon measure $\mu$ be $\alpha$-homogeneous and $p$-concave, where either $\alpha\geq n$ and $p>0$ or $\alpha<n$ but $p\in\left\{\frac{q}{\alpha},\frac{1}{\alpha}\right\}$. Then, if
$dS^{\mu}_{K,q}=dS^{\mu}_{L,q}$, one has $K=L$.
\end{lemma}
\begin{proof}
If $\alpha\geq n$, then, from Proposition~\ref{p:gal}, $\mu$ is also $1/\alpha$-concave. If $\alpha<n$, then $p\in\left\{\frac{q}{\alpha},\frac{1}{\alpha}\right\}$ by hypothesis; in either case we obtain from Corollary~\ref{q_min_p} that
$\frac{1}{q}\mu(K,K)=\mu_q(K,K)=\mu_q(L,K)\geq\frac{\alpha}{q}\left[\mu(L)^{1-\frac{q}{\alpha}}\mu(K)^{\frac{q}{\alpha}}\right].$
Multiplying by $t^{\alpha-1}$ and integrating over $(0,1)$ then yields
$\mu(K)^{1-\frac{q}{\alpha}}\geq \mu(L)^{1-\frac{q}{\alpha}}.$ From the symmetry of the expression, we also have that
$\mu(L)^{1-\frac{q}{\alpha}}\geq \mu(K)^{1-\frac{q}{\alpha}}.$ Therefore, $\mu(K)=\mu(L)$. Thus, $K=L$, from either Proposition~\ref{p:mr} or Proposition~\ref{p:RX}. 
\end{proof}

\noindent The following theorem extends the result of Wu \cite[Theorem 6.4]{DW17}, which had more restrictions on the concavity and no uniqueness result.

\begin{theorem}
\label{t:min_exist_uni_q}
Fix $q\geq 1$. Let a Radon measure $\mu$ be even, $p$-concave, and $\alpha$- homogeneous so that $\alpha \neq q$ and either $\alpha \geq n$ and $p>0$ or $\alpha<n$ but $p\in\left\{\frac{1}{\alpha},\frac{q}{\alpha}\right\}$. Suppose $\nu$ is a finite, even measure on $\s^{n-1}$ not concentrated on any hemisphere. Then, there exists a unique symmetric $K\in\conbod_0$ such that
$$d\nu=dS^{\mu}_{K,q}.$$
\end{theorem}

We now extend the measures under consideration. The following have proofs that are, up to minor adjustments, the same as Proposition~\ref{p:uni_0}, Remark~\ref{r:uni_1} and Lemma~\ref{l:uni}, and are omitted.
\begin{proposition}
\label{p:uni_0_1}
Fix $q\geq 1$. Let $\mu$, a Borel measure with continuous density, be $F$-concave such that $F$ is differentiable. Suppose $K,L\in\conbod_0$ such that $$dS^{\mu}_{K,q}=dS^{\mu}_{L,q}.$$ Then
$$\frac{F(\mu(L))-F(\mu(K))}{F^{\prime}(\mu(K))}\leq \frac{F(\mu(L))-F(\mu(K))}{F^{\prime}(\mu(L))}.$$
\end{proposition}
\begin{remark}
\label{r:uni_1_q}
If $\mu(K)=\mu(L)$ and $dS^{\mu}_{K,q}=dS^{\mu}_{L,q}$, we obtain equality in Minkowski's inequality. Thus, by Corollary~\ref{cor:minkowski_q}, we have equality in Equation~\ref{eq:fcon_q}.
\end{remark}

\begin{lemma}
\label{l:uni_q}
Fix $q\geq 1$. Let $\mu$, a Borel measure with continuous density, be $F$-concave such that $F$ is differentiable and there exists some $a\geq 0$ such that, for every $b,c\in [a,\mu(\R^n))$ one has $(F(b)-F(c))(F^\prime(b)-F^\prime(c))\geq 0$, $F^\prime(b)F^\prime(c)>0$, and $F^\prime(b)\neq F^\prime(c)$. Suppose $K,L\in\conbod_0$ such that $$dS^{\mu}_{K,q}=dS^{\mu}_{L,q}.$$ If $\mu(K),\mu(L)\geq a$, then $\mu(K)=\mu(L)$ and there is equality in Equation~\ref{eq:fcon_q}, which then yields equality in Equation~\ref{eq:concave}.
\end{lemma}

We conclude by applying Lemma~\ref{l:uni_q} with $F(x)=\Phi^{-1}(x)$, the Ehrhard function, to recover a result for the Gaussian measure established by Liu \cite{JL22}.
\begin{theorem}
\label{t:ehr_q}
Fix $q\geq 1$. Suppose $K,L\in\conbod_0$ such that $\gamma_n(K),\gamma_n(L)\geq 1/2$ and $$dS^{\gamma_n}_{K,q}=dS^{\gamma_n}_{L,q}.$$ Then, one has $K=L$.
\end{theorem}

\bibliography{references}
\bibliographystyle{amsplain}

\end{document}